\documentclass[12pt]{amsart}
\usepackage{a4wide,enumerate,xcolor,graphicx}
\usepackage{amsmath}
\usepackage{scrextend} 
\allowdisplaybreaks

\let\pa\partial
\let\na\nabla
\let\eps\varepsilon
\newcommand{\N}{{\mathbb N}}
\newcommand{\R}{{\mathbb R}}

\newcommand{\diver}{\operatorname{div}}

\newcommand{\V}{H^1_D(\Omega)}
\newcommand{\nD}{\overline{n}}
\newcommand{\pD}{\overline{p}}
\newcommand{\VD}{\overline{V}}
\newcommand{\DD}{\overline{D}}

\newtheorem{theorem}{Theorem}
\newtheorem{lemma}[theorem]{Lemma}

\begin{document}

\title[Drift-diffusion models for memristors]{Three-species drift-diffusion models \\
for memristors} 

\author[C. Jourdana]{Cl\'ement Jourdana}
\address{Laboratoire Jean Kuntzmann, 700 avenue centrale, 38400 Saint Martin 
d'H\`eres, France}
\email{clement.jourdana@univ-grenoble-alpes.fr} 

\author[A. J\"ungel]{Ansgar J\"ungel}
\address{Institute for Analysis and Scientific Computing, Vienna University of  
	Technology, Wiedner Hauptstra\ss e 8--10, 1040 Wien, Austria}
\email{juengel@tuwien.ac.at} 

\author[N. Zamponi]{Nicola Zamponi}
\address{Institute for Analysis and Scientific Computing, Vienna University of  
	Technology, Wiedner Hauptstra\ss e 8--10, 1040 Wien, Austria}
\email{nicola.zamponi@tuwien.ac.at}

\date{\today}

\thanks{The second and third authors acknowledge partial support from   
the Austrian Science Fund (FWF), grants P33010, W1245, and F65.
This work has received funding from the European 
Research Council (ERC) under the European Union's Horizon 2020 research and 
innovation programme, ERC Advanced Grant no.~101018153.} 

\begin{abstract}
A system of drift-diffusion equations for the electron, hole, and 
oxygene vacancy densities in a semiconductor,
coupled to the Poisson equation for the electric potential, is analyzed in a bounded
domain with mixed Dirichlet--Neumann boundary conditions. This system describes the
dynamics of charge carriers in a memristor device. 
Memristors can be seen as nonlinear resistors
with memory, mimicking the conductance response of biological synapses. 
In the fast-relaxation limit, the system reduces to a drift-diffusion system
for the oxygene vacancy density and electric potential, which is often used
in neuromorphic applications. The following results are proved:
the global existence of weak solutions to the full system in any
space dimension; the uniform-in-time boundedness of the solutions to the full system
and the fast-relaxation limit in two space dimensions; the global existence and weak-strong uniqueness analysis of the reduced system.
Numerical experiments in one space dimension illustrate the behavior of the
solutions and reproduce hysteresis effects in the current-voltage characteristics.
\end{abstract}

 

\maketitle


\section{Introduction}

The evolution of the microelectronics industry was influenced for more than 
50 years by Moore's law that predicts a doubling of the number of transistors on 
a microchip about every two years. As this observation is going to cease to 
apply because of physical scaling limitations, novel technologies or computing 
approaches are needed. Neuromorphic computing seems to be a promising avenue. 
It is a concept developed by C.~Mead in the late 1980s to implement aspects of 
(biological) neuronal networks as analog or digital copies on electric circuits. 

A promising device as technology enabler of neuromorphic computing is the memristor, 
which was postulated in \cite{Chu71}. We understand a memristor as a nonlinear 
resistor with memory showing a resistive switching behavior. For a historical debate 
of the memristor definition, we refer to \cite{VoMe15}. Artificial neurons and 
synapses can be built by using, e.g., ferroelectric materials, phase-change 
materials, or memristive materials \cite{IeAm20}. The oxide-based memristor consists 
of a thin titanium dioxide film between two metal electrodes \cite{Mla19}. The oxygen 
vacancies act as charge carriers. When an electric field is applied, the oxygen 
vacancies drift and change the boundary between the low- and high-resistance layers. 
In this way, memristors are able to mimic the conductance response of synapses. 
Advantages of these devices are the low power consumption, 
short switching time, and nano-size, allowing for high-density circuit architectures.

Memristor devices can be described by compact models, 
relating the charge and flux and using the memristor Ohm law \cite{Mla19}. 
In this paper, we are interested
in the internal physical processes of an oxide-based memristor, 
and we focus on diffusive models like those in \cite{GSTD13,SBW09}. 
They consist of drift-diffusion equations for the electron, hole, and oxygen
vacancy densities and the Poisson equation for the electric potential.

Since the electron-lattice relaxation is much faster than the oxygen vacancy drift, 
it is sufficient to determine the electron and hole densities from
the stationary equations, while the oxygen vacancy density still satisfies the
transient equation. In this paper, we make this limit rigorous.
More precisely, we prove the global existence of weak solutions to the
full transient model in any space dimension
and the fast-relaxation limit in two space dimensions.
Furthermore, we analyze the limiting model (existence, weak-strong uniqueness)
and present some finite-volume simulations in one space dimension.
Up to our knowledge, this is the first mathematical analysis of a charge transport
model for memristors.

\subsection{Model equations and mathematical difficulties}

The scaled equations for the electron density $n$, hole density $p$, 
oxygen vacancy density $D$ (or charged mobile $n$-type dopant density), 
and electric potential $V$ are given by
\begin{align}
  & \eps\pa_t n = \diver J_n, \quad J_n = \na n - n\na V, \label{1.n} \\
  & \eps\pa_t p = -\diver J_p, \quad J_p = -(\na p + p\na V), \label{1.p} \\
  & \pa_t D = -\diver J_D, \quad J_D = -(\na D + D\na V), \label{1.Nd} \\
  & \lambda^2\Delta V = n - p - D + A(x)\quad\mbox{in }\Omega,\ t>0, \label{1.V}
\end{align}
where $\eps>0$ is a small parameter describing the speed of relaxation to
the steady state, $\lambda>0$ is the (scaled) Debye length, 
$J_n$, $J_p$, and $J_D$ are the current densities of the electrons, holes, and
oxygen vacancies, respectively,
and $A(x)$ is the given immobile $p$-type dopant (acceptor) density. 
Following \cite{SBW09}, we neglect recombination-generation terms.
We use initial and physically motivated mixed Dirichlet--Neumann boundary conditions:
\begin{align}
  n(\cdot,0) = n^I, \quad p(\cdot,0) = p^I, \quad D(\cdot,0) = D^I
	&\quad\mbox{in }\Omega, \label{1.ic} \\
	n = \nD , \quad p = \pD , \quad V = \VD  
	&\quad\mbox{on }\Gamma_D,\ t>0,
	\label{1.dbc} \\
	J_n\cdot\nu = J_p\cdot\nu = \na V\cdot\nu = 0 &\quad\mbox{on }\Gamma_N,\ t>0,
	\label{1.nbc} \\
	J_D\cdot\nu = 0 &\quad\mbox{on }\pa\Omega,\ t>0. \label{1.nbcd}
\end{align}
This means that we prescribe the electron and hole densities as well as the
applied voltage on the Ohmic contacts $\Gamma_D$, while $\Gamma_N$
models the union of insulating boundary segments. 
The boundary is assumed to be not transparent to the oxygen vacancies, so
we assume no-flux conditions for $D$. This gives one of the mathematical
difficulties of the model, since we cannot perform certain partial integrations
as for $n$ and $p$. 

Another difficulty comes from the fact that we consider three species instead
two. Indeed, the quadratic drift terms can be estimated in the two-species system
for $(n,p,V)$
by exploiting a monotonicity property. Assuming for simplicity that $\nD=\pD=0$,
using $n$ and $p$ as test functions in the
weak formulations of \eqref{1.n} and \eqref{1.p}, respectively, 
and adding both equations, we find from \eqref{1.V} that
\begin{align}\label{1.est}
  \frac{\eps}{2}\frac{d}{dt}
	&\int_\Omega(n^2+p^2)dx + \int_\Omega(|\na n|^2+|\na p|^2)dx
  = \int_\Omega n\na V\cdot\na n dx - \int_\Omega p\na V\cdot\na p dx \\
	&= \frac12\int_\Omega\na(n^2-p^2)\cdot\na V dx
	= -\frac{1}{2\lambda^2}\int_\Omega(n^2-p^2)(n-p-D+A)dx \nonumber \\
	&\le -C(\lambda,D,A)\int_\Omega(n^2+p^2)dx, \nonumber
\end{align}
since $(n^2-p^2)(n-p)\ge 0$ and $D$ is fixed in the two-species model. 
This computation reduces the cubic term to
a quadratic one, which can be treated by Gronwall's lemma. This idea cannot
be applied to the three-species model.

\subsection{State of the art and strategy of our proofs}

These difficulties explain why there are only few analytical results in the 
literature on $n$-species drift-diffusion equations with $n>2$. 
They have been derived in \cite{WLL15} from a kinetic
Vlasov--Poisson--Fokker--Planck system in the diffusion limit.
In \cite{BBCFHT12}, a three-species system similar to ours is considered,
in the context of corrosion models, but only a stability analysis of a
finite-volume scheme has been performed. The authors of \cite{VPSS18}
analyze a four-species system, but their model includes drift terms only
in the equations for the electrons and holes, which enables the authors to use the
monotonicity property explained above. General existence results for an $n$-species
model have been proved in \cite{HPR19} for an abstract drift operator
imposing suitable smoothing conditions.
Estimates in Lebesgue and H\"older spaces for $n$-species systems 
have been derived in \cite{ChLu95} without an existence analysis.
More general models involving positive semidefinite, nondiagonal 
mobility matrices can be found in, e.g., \cite{DDGG20}.
A global existence analysis for $n$-species models was performed in
\cite{BFS14,GlHu97,GlHu05} (and the large-time asymptotics in \cite{GGH16})
assuming at most two space dimensions. This restriction can be understood as follows.

Instead of integrating by parts as in \eqref{1.est}, the idea is to use
an elliptic estimate for $V$. Because of the mixed boundary conditions, we cannot
expect full elliptic regularity for the Poisson equation, but
there exists $r_0>2$ such that \cite{Gro94}
\begin{equation}\label{1.regulV}
  \|\na V\|_{L^{r_0}(\Omega)}\le C\big(1+\|n-p-D+A\|_{L^{2r_0/(r_0+2)}(\Omega)}\big);
\end{equation}
see Lemma \ref{lem.vinfty} in the Appendix for the precise statement. 
Using $\log n-\log\nD$ as a test function in the weak formulation of \eqref{1.n},
we can derive a uniform estimate for $n\log n$ in $L^\infty(0,T;L^1(\Omega))$;
see \eqref{1.dHdt} below.
Then the H\"older inequality with $r_0'=2r_0/(r_0-2)$ and a generalized 
Gagliardo--Nirenberg inequality (see Lemma \ref{lem.GN} below) lead to 
\begin{align*}
  \int_\Omega & n\na V\cdot\na n dx \le \|n\|_{L^{r_0'}(\Omega)}
	\|\na V\|_{L^{r_0}(\Omega)}	\|\na n\|_{L^2(\Omega)} \\
	&\le C\|n\|_{L^{r'_0}(\Omega)}\big(1+\|n-p-D+A\|_{L^{2r_0/(r_0+2)}(\Omega)}\big)
	\|\na n\|_{L^2(\Omega)} \\
	&\le \delta\|\na n\|_{L^2(\Omega)}^{1+2d/(d+2)} + C(n,p,D,\delta),
\end{align*}
where $\delta>0$, and $C(n,p,D,\delta)>0$ depends on the $L^1\log L^1$ norms of 
$n$, $p$, and $D$. The first term on the left-hand side can be absorbed, for
sufficiently small $\delta>0$, by the gradient term coming from the diffusion part
if the exponent is not larger than two, and this is the case if and only if $d\le 2$.

Our strategy is different. As in \cite{GlHu97}, the key estimate comes from the
free energy functional
\begin{align}\label{1.H}
  H[n,p,D,V] &= \int_\Omega\bigg\{n\bigg(\log\frac{n}{\nD }-1\bigg)
	+ p\bigg(\log\frac{p}{\pD }-1\bigg) + D(\log D-1+\VD )\bigg\}dx \\
	&\phantom{xx}{}+ \frac{\lambda^2}{2}\int_\Omega|\na(V-\VD)|^2 dx. \nonumber
\end{align}
The first integral models the thermodynamic entropy, while the second integral
corresponds to the electric energy. We prove in Theorem \ref{thm.glob} that
\begin{equation}\label{1.dHdt}
  \frac{dH}{dt} + \int_\Omega\bigg(\frac{n}{2\eps}|\na(\log n-V)|^2
	+ \frac{p}{2\eps}|\na(\log p+V)|^2 + \frac{D}{2}|\na(\log D+V)|^2\bigg)dx
	\le C(\nD,\pD,\VD).
\end{equation}
While the authors of \cite{GlHu97} have used this free energy inequality
as the starting point to derive iteratively $L^\infty$ estimates in two
space dimensions, we use
another argument that allows us to obtain a global existence result in any
space dimension. 

More precisely, we prove that \eqref{1.dHdt} implies an $L^2(0,T;W^{1,1}(\Omega))$
bound for $\sqrt{n_k}$ (as well as $\sqrt{p_k}$ and $\sqrt{D_k})$,
where $(n_k,p_k,D_k,V_k)$ is a solution to an approximate problem with $k\in\N$.
This bound is not sufficient to deduce strong compactness.
By a cutoff argument, we show that $n_k$ (as well as $p_k$ and $D_k$) are bounded
in $L^r(0,T;W^{1,r}(\Omega_\delta))$, where $\Omega_\delta=\{x\in\Omega:
\operatorname{dist}(x,\pa\Omega)>\delta\}$ and $r>1$. By the Aubin--Lions lemma, we
conclude strong $L^s(\Omega_\delta\times(0,T))$ convergence of $n_k$ for $s<r$, and
by the Theorem of de la Vall\'ee--Poussin, weak $L^1(\Omega\times(0,T))$ convergence
of $n_k$. Then we deduce from a Cantor diagonal argument the strong 
$L^1(\Omega\times(0,T))$ convergence of $n_k$ (as well as $p_k$ and $D_k$).
This is the key argument to prove the global existence of weak solutions to
\eqref{1.n}--\eqref{1.nbcd} in any space dimension. Our
strategy extends the results of \cite{BFS14,GlHu97,GlHu05}.

The second main result of this paper is the fast-relaxation limit $\eps\to 0$ 
for the solutions $(n_\eps,p_\eps,D_\eps,V_\eps)$ to 
\eqref{1.n}--\eqref{1.nbcd}. We expect that $J_n=J_p=0$ holds in the limit,
leading to $n_\eps=c_n e^{V_\eps}$ and $p_\eps=c_p e^{-V_\eps}$, where 
$c_n$, $c_p>0$ are constants determined by the Dirichlet boundary data. 
The limit $\eps\to 0$ was already 
performed in a two-species drift-diffusion system \cite{JuPe00}, 
exploiting a uniform lower positive bound for $n_\eps$ and $p_\eps$. 
Unfortunately, this argument cannot be used for our three-species system, 
and we need another idea.

The starting point is again the free energy inequality \eqref{1.dHdt}, showing that
$$
  \sqrt{n_\eps}\na(\log n_\eps-V_\eps)\to 0, \quad
	\sqrt{p_\eps}\na(\log p_\eps+V_\eps)\to 0
$$
strongly in $L^2(Q_T)$ as $\eps\to 0$, where $Q_T=\Omega\times(0,T)$. 
Since equation \eqref{1.Nd} for $D_\eps$ does not contain $\eps$, we obtain
$D_\eps\to D_0$ strongly in $L^1(Q_T)$ from the Aubin--Lions lemma.
As in \cite{JuPe00}, the key step is to prove the strong convergence of
$\na V_\eps$, but in contrast to that work, we are lacking some estimates.
We formulate the Poisson equation for $V_\eps$ as
$$
  \lambda^2\Delta V_\eps = c_n e^{V_\eps} - c_p e^{-V_\eps} - D_\eps + A(x) 
  + E_\eps,
$$
where $E_\eps$ is an error term.
Similar as in \cite{JuPe00}, we exploit the monotonicity of $V_\eps\mapsto
c_n e^{V_\eps} - c_p e^{-V_\eps}$ to prove that $(\na V_\eps)$ is a Cauchy sequence
and hence convergent in $L^2(Q_T)$. The novelty is the proof of
$E_\eps\to 0$ as $\eps\to 0$. Here, we need an $L^\infty(Q_T)$ bound for $V_\eps$,
and this is possible (only) in two space dimensions, according to \cite{Gro94}:
$$
  \|V_\eps\|_{L^\infty(\Omega)} \le C\big(1+\|(n_\eps-p_\eps-D_\eps+A)
	\log|n_\eps-p_\eps-D_\eps+A|\|_{L^1(\Omega)}\big) \le C.
$$
We infer that $\na V_\eps\to\na V_0$ strongly in $L^2(Q_T)$,
and $V_0$ solves the limiting Poisson equation 
$\lambda^2\Delta V_0 = c_n e^{V_0} - c_p e^{-V_0} - D_0 + A(x)$.
Note that, in contrast to \cite{JuPe00},
we need the restriction to two space dimensions.

\subsection{Main results}

We impose the following assumptions.

\begin{itemize}
\item[(A1)] Domain: $\Omega\subset\R^d$ ($d\ge 1$) is a bounded domain with Lipschitz
boundary $\pa\Omega=\Gamma_D\cup\Gamma_N$, $\operatorname{meas}(\Gamma_D)>0$,
and $\Gamma_N$ is relatively open in $\pa\Omega$. 
\item[(A2)] Data: $T>0$, $\eps>0$, $\lambda>0$, $A\in L^\infty(\Omega)$. 
\item[(A3)] Boundary data: $\nD$, $\pD$, $\VD\in W^{1,\infty}(\Omega)$ satisfy
$\nD$, $\pD>0$ in $\Omega$.
\item[(A4)] Initial data: $n^I$, $p^I$, $D^I\in L^2(\Omega)$ 
satisfy $n^I$, $p^I$, $D^I\ge 0$ in $\Omega$.
\end{itemize}

We set $Q_T=\Omega\times (0,T)$,
$H_D^1(\Omega)=\{u\in H^1(\Omega):u=0$ on $\Gamma_D\}$,
and we introduce the initial electric potential $V^I-\VD\in H_D^1(\Omega)$ as the 
unique solution to
\begin{align*}
  & \lambda^2\Delta V^I = n^I-p^I-D^I+A\quad\mbox{in }\Omega, \\
	& V^I = \VD  \quad\mbox{on }\Gamma_D, \quad \na V^I\cdot\nu=0\quad\mbox{on }\Gamma_N.
\end{align*}
The boundary data in Assumption (A3) are supposed to be time-independent to
simplify the computations.
In two space dimensions, it is sufficient to assume in Assumption (A4) that
$n^I\log n^I$, $p^I\log p^I$, $D^I\log D^I\in L^1(\Omega)$ since
\cite[Lemma 2.2]{Gro96} implies that $n^I-p^I-D^I+A\in H_D^1(\Omega)'$.
The regularity conditions in Assumptions (A3) and (A4) are imposed for simplicity;
they can be slightly weakened. 

Our first main result is the global existence of weak solutions in any
space dimension. 

\begin{theorem}[Global existence]\label{thm.glob}
Let Assumptions (A1)--(A4) hold. Then there exists a weak solution $(n,p,D,V)$
to \eqref{1.n}--\eqref{1.nbcd} satisfying
\begin{align*}
  & n\log n,\,p\log p,\,D\log D\in L^\infty(0,T;L^1(\Omega)), \\
	& \sqrt{n},\,\sqrt{p},\,\sqrt{D}\in L^2(0,T;W^{1,1}(\Omega)), \quad
	J_n,\,J_p,\,J_D\in L^1(Q_T), \\
	& \pa_t n,\,\pa_t p,\in L^1(0,T;X'), \quad \pa_t D\in L^1(0,T;H^{s}(\Omega)'),
	\quad V\in L^\infty(0,T;H^1(\Omega)),
\end{align*}
where $X:=H^s(\Omega)\cap H_D^1(\Omega)$ and $s>1+d/2$. 
This solution satisfies the free energy inequality
\begin{align}\label{1.ei}
  H[(n,p,D,V)(t)] &+ \int_0^t\int_\Omega\bigg(\frac{n}{2\eps}
	|\na(\log n-V)|^2 + \frac{p}{2\eps}|\na(\log p+V)|^2 \\
	&{}+ \frac{D}{2}|\na(\log D+V)|^2\bigg)dx ds
	\le H^I + C(H^I,\Lambda_\eps,T), \nonumber
\end{align}
where the initial free energy $H^I[n,p,D,V]$ is defined in \eqref{1.H}, 
$H^I := H[n^I,p^I,D^I,V^I]$, 
\begin{equation}\label{1.lambda}
	\Lambda_\eps := \frac{1}{2\eps}\big(\|\na(\log \nD -\VD )\|_{L^\infty(Q_T)}^2
	+ \|\na(\log \pD +\VD )\|_{L^\infty(Q_T)}^2\big),
\end{equation}
and it holds that $C(H^I,\Lambda_\eps,T)=0$ if $\Lambda_\eps=0$.
\end{theorem}

The property $\Lambda_\eps=0$ means that the boundary data are in thermal
equilibrium. In this case, the free energy is a nonincreasing function of time.
The entropy production in \eqref{1.ei} is understood in the sense
$n|\na(\log n-V)|^2 = |2\na\sqrt{n}-\sqrt{n}\na V|^2$, i.e.. we have
$2\na\sqrt{n}-\sqrt{n}\na V\in L^2(Q_T)$. 

We approximate \eqref{1.n}--\eqref{1.V} by truncating the drift term and
proving the existence of a solution $(n_k,p_k,D_k,V_k)$ to the approximate
problem. Estimates uniform in the truncation parameter $k$ are obtained from
an approximate free energy inequality, similar to \eqref{1.dHdt}. As explained
before, we also derive uniform estimates in the domain $\Omega_\delta$,
which are needed to conclude the strong $L^1(Q_T)$ convergence of the approximate
solution. Because of low regularity, the difficulty is to identify the weak
limit of a truncated version of $2\na\sqrt{n_k}-\sqrt{n_k}\na V_k$. This is 
done by combining the free energy estimates and the Aubin--Lions lemma,
applied in the domain $\Omega_\delta\times(0,T)$.

Similarly, as in \cite{GlHu97}, we can prove, in the two-dimensional case,
that the weak solution from
Theorem \ref{thm.glob} is bounded uniformly in time. 

\begin{theorem}[Uniform $L^\infty$ bounds]\label{thm.infty}
Let (A1)--(A4) hold, let $d\le 2$, and let 
$n^I$, $p^I$, $D^I\in L^\infty(\Omega)$.
Furthermore, let $(n,p,D,V)$ be the weak solution to 
\eqref{1.n}--\eqref{1.nbcd} constructed in Theorem \ref{thm.glob}.
Then there exists a constant $C(\eps)>0$ depending on $\eps$
such that for all $t>0$,
$$
  \|n(t)\|_{L^\infty(\Omega)} + \|p(t)\|_{L^\infty(\Omega)} 
	+ \|D(t)\|_{L^\infty(\Omega)} \le C(\eps).
$$
\end{theorem}

The theorem is proved by an Alikakos-type iteration method.
The restriction to two space dimensions comes from the regularity
\eqref{1.regulV} for the electric potential. The rough idea of the proof is
to choose $n_k^{q-1}$ in the weak formulation of \eqref{1.n} (and similarly
for $p_k$ and $D_k$) and to derive an estimate in $L^q(Q_T)$, which is uniform
in $k$ and $q$. Then the limit $k$, $q\to\infty$ gives the desired $L^\infty(Q_T)$ 
bound. Since $n_k^{q-1}$ is generally not an $H^1(\Omega)$ function for $q>2$, 
we prove first that a truncated version of $n_k^{q-1}$ lies in $L^\infty(Q_T)$,
but possibly not uniformly in $k$. Next, we choose $e^t\max\{0,n_k-M\}^{q-1}$
for sufficiently large $M>0$ as a test function, show that $n_k\in L^q(Q_T)$ 
uniformly in $k$, $q$, and $T$, and pass to the limit $k$, $q\to\infty$. 
The factor $e^t$ is needed to obtain a time-uniform estimate.

Next, we study the limit problem, which is formally obtained by setting
$\eps=0$ in \eqref{1.n}--\eqref{1.V} and taking into account the Dirichlet data:
\begin{align}
  & \pa_t D_0 = \diver(\na D_0+D_0\na V_0), \label{1.limD} \\ 
	& \lambda^2\Delta V_0 = c_n e^{V_0} - c_p e^{-V_0} - D_0 + A(x) 
	\quad\mbox{in }\Omega,\ t>0, \label{1.limV} \\
	& (\na D_0+D_0\na V_0)\cdot\nu=0\quad\mbox{on }\pa\Omega, \quad
	D_0(\cdot,0) = D^I\quad\mbox{in }\Omega, \label{1.limbcD} \\
	& V_0 = \VD\quad\mbox{on }\Gamma_D, \quad \na V_0\cdot\nu = 0 
	\quad\mbox{on }\Gamma_N, \label{1.limbcV}
\end{align}
where $c_n=\nD\exp(-\VD)$, $c_p=\pD\exp(\VD)$, and the electron and hole densities
are determined by $n_0 = c_n\exp(V_0)$ and $p_0 = c_p\exp(-V_0)$, respectively.
We show the global existence of weak solutions and
verify a weak-strong uniqueness property.

\begin{theorem}[Existence and weak-strong uniqueness for the limit problem]
\label{thm.lim}
Let Assumptions (A1)--(A4) hold and let $\nD\ge c>0$, $\pD\ge c>0$ in $\Omega$.
Then there exists a bounded weak solution
$(D_0,V_0)$ to \eqref{1.limD}--\eqref{1.limbcV}. Moreover, 
if $(D,V)$ is a weak solution and
$(D_0,V_0)$ a bounded strong solution to \eqref{1.limD}--\eqref{1.limV}
satisfying
$$
  \inf_{Q_T}D_0>0, \ D_0, \na\log D_0, V_0, \na V_0\in L^\infty(Q_T), \
	\pa_t D_0,\pa_t V_0\in L^1(0,T;L^\infty(\Omega)),
$$
then $D=D_0$, $V=V_0$ in $\Omega\times(0,T)$.
\end{theorem}

For the proof of the existence of a weak solution to 
\eqref{1.limD}--\eqref{1.limbcV}, we use the techniques of the proof of
Theorem \ref{thm.glob}. Let $(D_{0,k},V_{0,k})$ be a solution to a
truncated problem. The approximate free energy inequality gives us only
the weak convergence of $V_{0,k}$, which is not sufficient to perform 
the limit $k\to\infty$ in the nonlinear Poisson equation.
We need the strong convergence of $V_{0,k}$. Our idea is to derive
first an $L^1(Q_T)$ bound for $V_{0,k}\exp(|V_{0,k}|)$, which follows
from the free energy inequality for the reduced model or directly from
the nonlinear Poisson equation. Second, we
prove that $(V_{0,k})$ is a Cauchy sequence. This is done by taking a particular
nonlinear test function in the Poisson equation, satisfied by the difference
$V_{0,k}-V_{0,\ell}$, which leads to
\begin{align*}
 \lambda^2&\int_0^T\int_\Omega\frac{2+(V_{0,k}-V_{0,\ell})^2}{2
	(1+(V_{0,k}-V_{0,\ell})^2)^{5/4}}|\na(V_{0,k}-V_{0,\ell})|^2 dxdt \\
	&\phantom{xx}{}+ C\int_0^T\int_\Omega
    \frac{(V_{0,k}-V_{0,\ell})(\sinh(V_{0,k})-\sinh(V_{0,\ell}))}{
    (1+(V_{0,k}-V_{0,\ell})^2)^{1/4}}dxdt \\
	&\le  \int_0^T\int_\Omega F(V_{0,k}-V_{0,\ell},D_{0,k},D_{0,\ell})dxdt
\end{align*}
for some constant $C>0$, where $F$ is some function; 
see Section \ref{sec.prooflim} for details. 
Using the Fenchel--Young inequality and
the De la Valle\'e--Poussin theorem, 
the right-hand side is shown to converge to zero as $k$, $\ell\to\infty$.
Then the properties of the hyperbolic sine function prove the claim.

The weak-strong uniqueness property is based on an estimation of the
relative free energy
\begin{align*}
  H[(D,V)|(D_0,V_0)] &= \int_\Omega\bigg(\frac{\lambda^2}{2}|\na(V-V_0)|^2
	 + c_n e^{V_0}f_0(V-V_0) + c_p e^{-V_0}f_0(V_0-V)\bigg)dx \\
	&\phantom{xx}{}+ \int_\Omega\bigg(D\log\frac{D}{D_0}-D+D_0\bigg)dx,
\end{align*}
where $f_0(s)=(s-1)e^s+1$ for $s\in\R$. The idea is to show that
$$
  \frac{dH}{dt}[(D,V)|(D_0,V_0)] + \frac12\int_\Omega
	D\bigg|\na\bigg(\log\frac{D}{D_0}+V-V_0\bigg)\bigg|^2 dx
	\le \gamma(t)H[(D,V)|(D_0,V_0)]
$$
for some $\gamma\in L^1(0,T)$ depending on the regularity of $(D_0,V_0)$.
By Gronwall's lemma, $H[(D,V)(t)|(D_0,V_0)(t)]=0$, proving that
$(D,V)(t)=(D_0,V_0)(t)$ for $t\in (0,T)$.

Our final main result is the fast-relaxation limit $\eps\to 0$.

\begin{theorem}[Limit $\eps\to 0$]\label{thm.eps}
Let $d\le 2$, $\nD =c_n\exp(\VD)$, and $\pD =c_p\exp(-\VD)$
in $\Omega$ for some positive constants $c_n$ and $c_p$. Let 
$(n_\eps,p_\eps,D_\eps,V_\eps)$ be a weak solution to \eqref{1.n}--\eqref{1.nbcd}
and $(n_0,p_0,D_0,V_0)$ be a weak solution to \eqref{1.limD}--\eqref{1.limV}. 
Then there exists a subsequence such that, as $\eps'\to 0$,
\begin{align*}
  & n_{\eps'}\to n_0,\quad p_{\eps'}\to p_0, \quad D_{\eps'}\to D_0
	\quad\mbox{strongly in }L^1(Q_T), \\
	& \na D_{\eps'}+D_{\eps'}\na V_{\eps'}\rightharpoonup \na D_0+D_0\na V_0
	\quad\mbox{weakly in }L^1(Q_T), \\ 
    & \pa_t D_{\eps'}\rightharpoonup \pa_t D_0
	\quad\mbox{weakly in }L^1(0,T;H^{s}(\Omega)'), \\	
    & V_{\eps'}\to V_0 \quad\mbox{strongly in }L^2(0,T;H^1(\Omega)),
\end{align*}
where $s>1+d/2$,
and $(D_0,V_0)$ is a weak solution to \eqref{1.limD}--\eqref{1.limbcV}.
\end{theorem}

If the limit problem $(D_0,V_0)$ is uniquely solvable, we achieve the convergence
of the whole sequence. The uniqueness of bounded weak solutions can be proved 
under regularity conditions on the electric potential (e.g.\ $\na V_0\in L^\infty$;
see \cite{Jue97}). However, this regularity cannot generally be expected for
mixed Dirichlet--Neumann boundary conditions. 

The paper is organized as follows. Theorems \ref{thm.glob}, \ref{thm.infty},
\ref{thm.lim}, and \ref{thm.eps} are proved in Sections \ref{sec.glob},
\ref{sec.infty}, \ref{sec.lim}, and \ref{sec.eps}, respectively. 
Some numerical experiments in one space dimension are performed in Section
\ref{sec.num}. Finally,
Appendix \ref{sec.aux} is concerned with the proof of some properties for
the truncation functions, and we recall some 
auxiliary results used in this paper.


\section{Proof of Theorem \ref{thm.glob}}\label{sec.glob}

In this section, we prove the global existence of weak solutions to
\eqref{1.n}--\eqref{1.nbcd}. First, we show the existence of solutions
to an approximate problem, derive some uniform estimates, and then pass to
the de-regularization limit. 

\subsection{Approximate problem for \eqref{1.n}--\eqref{1.nbcd}}

We define the approximate problem by truncating the nonlinear drift terms.
For this, we introduce the truncation 
$$
  T_k(s) = \max\{0,\min\{k,s\}\}\quad\mbox{for }s\in\R,\ k\ge 1,
$$
and define the approximate problem
\begin{align}
  \eps\pa_t n_k &= \diver(\na n_k-T_k(n_k)\na V_k), \label{3.nk} \\
	\eps\pa_t p_k &= \diver(\na p_k+T_k(p_k)\na V_k), \label{3.pk} \\
	\pa_t D_k &= \diver(\na D_k+T_k(D_k)\na V_k), \label{3.Ndk} \\
	\lambda^2\Delta V_k &= n_k-p_k-D_k + A\quad\mbox{in }\Omega,\ t>0,
	\label{3.Vk}
\end{align}
supplemented by the initial and boundary conditions
\begin{align}
  n_k(\cdot,0)=n^I, \ p_k(\cdot,0)=p^I, \ D_k(\cdot,0)=D^I
	&\quad\mbox{in }\Omega, \label{3.ic} \\
  n_k=\nD , \ p_k = \pD , \ V_k = \VD &\quad\mbox{on }\Gamma_D,\ t>0, 
	\label{3.dbc} \\
  \na n_k\cdot\nu = \na p_k\cdot\nu = \na V_k\cdot\nu = 0
	&\quad\mbox{on }\Gamma_N,\ t>0, \label{3.nbc} \\
  (\na D_k+T_k(D_k)\na V_k)\cdot\nu = 0 &\quad\mbox{on }\pa\Omega,\ t>0.
	\label{3.nbcd}
\end{align}

\subsection{Existence of solutions to the approximate problem}

We prove that the approximate problem has a weak solution.

\begin{lemma}\label{lem.exk}
Let Assumptions (A1)--(A4) hold. Then there exists a weak solution 
$(n_k,p_k,D_k,$ $V_k)$
to \eqref{3.nk}--\eqref{3.nbcd} satisfying 
$n_k\ge 0$, $p_k\ge 0$, $D_k\ge 0$ in $Q_T=\Omega\times(0,T)$ and
\begin{align*}
  & n_k,\,p_k,\,D_k,\,V_k\in L^2(0,T;H^1(\Omega)), \\
	& \pa_t n_k,\,\pa_t p_k\in L^2(0,T;\V'), \quad
	\pa_t D_k\in L^2(0,T;H^1(\Omega)').
\end{align*}
\end{lemma}

As a consequence of the lemma, $T_k(n_k)=\min\{k,n_k\}$ and similarly for $p_k$
and $D_k$.

\begin{proof}
The existence of weak solutions can be proved in a standard way by the
Leray--Schauder fixed-point theorem. Therefore, we only sketch the proof.
Let $(n^*,p^*,D^*)\in L^2(Q_T;\R^3)$ and $\sigma\in[0,1]$. 
The linear system
\begin{align*}
  \eps\pa_t n &= \diver(\na n-\sigma T_k(n^*)\na V), \\
	\eps\pa_t p &= \diver(\na p+\sigma T_k(p^*)\na V), \\
	\pa_t D &= \diver(\na D+\sigma T_k(D^*)\na V), \\
	\lambda^2\Delta V &= n-p-D + A\quad\mbox{in }\Omega,\ t>0,
\end{align*}
together with initial and boundary conditions \eqref{1.nbc}--\eqref{1.nbcd} and
\begin{align*}
  & n(\cdot,0) = \sigma n^I, \quad p(\cdot,0) = \sigma p^I, \quad 
	D(\cdot,0) = \sigma D^I	\quad\mbox{in }\Omega, \\
	& n = \sigma \nD , \quad p = \sigma \pD , \quad V = \sigma \VD  
	\quad\mbox{on }\Gamma_D,\ t>0,
\end{align*}
possesses a unique solution $(n,p,D,V)\in L^2(Q_T;\R^4)$. This defines the
fixed-point operator $F:L^2(Q_T;\R^3)\times[0,1]\to L^2(Q_T;\R^3)$,
$(n^*,p^*,D^*;\sigma)\mapsto(n,p,D)$. It holds that $F(n,p,D;0)=0$, and $F$ is
continuous. Both the compactness of $F$ and a $\sigma$-uniform bound
on the set of fixed points $F(n,p,D;\sigma)=(n,p,D)$ follow from
energy-type estimates and the Aubin--Lions lemma. Indeed, let $(n,p,D;\sigma)$
be a fixed point of $F(\cdot;\sigma)$, i.e.\ a solution to \eqref{3.nk}--\eqref{3.nbcd}.
We use the test function $V-\VD $ in the weak formulation of \eqref{3.Vk} 
and apply Young's and Poincar\'e's inequalities to find that for any $\delta>0$,
\begin{align*}
  \lambda^2\int_\Omega|\na(V-\VD)|^2 dx
	&= -\lambda^2\int_\Omega\na\VD\cdot\na(V-\VD) dx
	- \int_\Omega(n-p-D+A)(V-\VD)dx \\
	&\le \frac{\lambda^2}{2}\int_\Omega|\na (V-\VD)|^2 dx
	+ \frac{\lambda^2}{2}\int_\Omega|\na \VD |^2 dx
	+ \delta C\int_\Omega|\na(V-\VD )|^2 dx \\
	&\phantom{xx}{}+ C(\delta)\int_\Omega(n-p-D+A)^2 dx,
\end{align*}
and choosing $\delta>0$ sufficiently small and integrating over $(0,T)$ gives
$$
  \int_0^T\int_\Omega|\na V|^2 dxds
  \le C + C\int_0^{T}\int_\Omega(n^2+p^2+D^2)dxds,
$$
where $C>0$ denotes here and in the following a generic constant independent
of $\eps$ with values changing from line to line.

Next, we use the test function $n-\nD $ in the weak formulation of \eqref{3.nk} and
use $T_k(n)\le k$:
\begin{align*}
  \frac{\eps}{2}\int_\Omega&(n(t)-\nD )^2 dx - \frac{\eps}{2}\int_\Omega(n^I-\nD )^2dx
	+ \int_0^t\int_\Omega \na n\cdot\na(n-\nD) dxds \\
	&= \sigma\int_0^t\int_\Omega T_k(n)\na V\cdot\na(n-\nD )dxds \\
	&\le \delta\int_0^t\int_\Omega|\na(n-\nD )|^2 dxds
	+ C(\delta)k^2\int_0^t\int_\Omega|\na V|^2 dxds.
\end{align*}
We deduce from the estimate for $\na V$ and some elementary manipulations that
$$
  \eps\int_\Omega n(t)^2 dx + \int_0^t\int_\Omega|\na n|^2 dxds
	\le C(k) + C(k)\int_0^t\int_\Omega(n^2+p^2+D^2)dxds.
$$
Using $p-\pD $ and $D$ as test functions in the weak formulations of
\eqref{3.pk} and \eqref{3.Ndk}, respectively, and estimating as above, 
we conclude that
\begin{align*}
  \int_\Omega&\big(\eps n(t)^2 + \eps p(t)^2 + D(t)^2\big) dx 
	+ \int_0^t\int_\Omega\big(|\na n|^2 + |\na p|^2 + |\na D|^2\big)dxds \\
	&\le C(k) + C(k)\int_0^t\int_\Omega(n^2+p^2+D^2)dxds. 
\end{align*}
Gronwall's lemma yields $\sigma$-uniform bounds for $n$, $p$, $D$, $V$ in
$L^2(0,T;H^1(\Omega))$. From these estimates, we can derive uniform bounds for
$\pa_t n$, $\pa_t p$ in $L^2(0,T;H_D^1(\Omega)')$ 
and for $\pa_t D$ in $L^2(0,T;H^1(\Omega)')$. 
These estimates are sufficient to apply the
Aubin--Lions lemma, which yields the compactness of the fixed-point operator
in $L^2(Q_T;\R^3)$ and allows us to apply the Leray--Schauder fixed-point theorem.

The nonnegativity of the densities follows directly after using $(n_k)_-=\min\{0,n_k\}$
as a test function in the weak formulation of \eqref{3.nk}, since $T_k(n_k)(n_k)_-=0$. 
The nonnegativity of $p_k$ and $D_k$ follows similarly. This finishes the proof.
\end{proof}

\subsection{Uniform estimates}

We wish to derive some $k$-uniform bounds using the free energy \eqref{1.H}. 
As the densities are only nonnegative, we cannot use $\log n_k$ etc.\ 
as a test function, and we need to regularize \eqref{1.H}.
For this, we introduce the function
$$
  G_{k,\delta}(s,\overline{s}) = g_{k,\delta}(s) - g_{k,\delta}(\overline{s}) 
	- g_{k,\delta}'(\overline{s})(s-\overline{s}), \quad\mbox{where }
	g_{k,\delta}(s) = \int_0^s\int_1^y \frac{dzdy}{T_k(z)+\delta},
$$
$s$, $\overline{s}\ge 0$, $k\ge 1$, $\delta >0$, and the regularized free energy
$$
  H_{k,\delta}[n,p,D,V] = \int_\Omega\bigg(G_{k,\delta}(n,\nD )
	+ G_{k,\delta}(p,\pD ) + G_{k,\delta}(D,\DD )
	+ \frac{\lambda^2}{2}|\na(V-\VD )|^2\bigg)dx,
$$
where $\DD$ is uniquely defined by $g'_{k,\delta}(\DD)=-\VD$.
The number $\DD$ depends on $k$ and $\delta$, but a computation shows that $\DD$
can be uniformly bounded with respect to $\delta$.
The function $g_{k,\delta}$ is constructed in such a way that the chain rule
$(T_k(n_k)+\delta)\na g_{k,\delta}'(n_k)=\na n_k$ is fulfilled.
An elementary computation shows that there exists $c>0$, not depending on $k$
and $\delta$, such that
$g_{k,\delta}(s)\ge c(s-1)$ for all $s\ge 0$. This implies that
\begin{equation}\label{3.Hk}
  H_{k,\delta}[n,p,D,V] \ge -C + C\int_\Omega(n+p+D)dx.
\end{equation}
For the next lemma, we define
\begin{align*}
  \Lambda_{\eps,k,\delta} 
	&= \frac{1}{2\eps}\big(\|\na(g_{k,\delta}'(\nD )-\VD)\|_{L^\infty(Q_T)}^2
	+ \|\na(g_{k,\delta}'(\pD )+\VD )\|_{L^\infty(Q_T)}^2\big), \\
	h_{k,\delta}(s) &= \int_0^s\frac{dy}{\sqrt{T_k(y)+\delta}}, \quad s\in\R, \\
	H_{k,\delta}^I &= H_{k,\delta}[n^I,p^I,D^I,V^I].
\end{align*}
The function $h_{k,\delta}$ satisfies the chain rule
$\sqrt{T_k(n_k)+\delta}\na h_{k,\delta}(n_k)=\na n_k$.

\begin{lemma}[Regularized free energy inequality I]
Let $(n_k,p_k,D_k,V_k)$ be a weak solution to the approximate problem
\eqref{3.nk}--\eqref{3.nbcd}. Then there exists a constant 
$C(H_{k,\delta}^I,\Lambda_{\eps,k,\delta},T,\delta)>0$ such that for all $0<t<T$,
\begin{align}\label{3.eik1}
  &H_{k,\delta}[(n_k,p_k,D_k,V_k)(t)]
	+ \frac{1}{2\eps}\int_0^t\int_\Omega
	\big|\na h_{k,\delta}(n_k)-\sqrt{T_k(n_k)+\delta}\na V_k\big|^2 dxds \\
	&\phantom{xx}{}
	+ \frac{1}{2\eps}\int_0^t\int_\Omega 
	\big|\na h_{k,\delta}(p_k)+\sqrt{T_k(p_k)+\delta}\na V_k\big|^2dxds \nonumber \\
	&\phantom{xx}{}
	+ \frac{1}{2} \int_0^t\int_\Omega\big|\na h_{k,\delta}(D_k)
	+\sqrt{T_k(D_k)+\delta}\na V_k\big|^2	dxds 
  \le H_{k,\delta}^I + C(H_{k,\delta}^I,\Lambda_{\eps,k,\delta},T,\delta), \nonumber
\end{align}
and the constant $C(H_{k,\delta}^I,\Lambda_{\eps,k,\delta},T,\delta)$ vanishes if 
$\Lambda_{\eps,k,\delta}=0$ and $\delta=0$.
\end{lemma}

\begin{proof}
We choose the test functions $g_{k,\delta}'(n_k)-g_{k,\delta}'(\nD )$, 
$g_{k,\delta}'(p_k)-g_{k,\delta}'(\pD )$, and
$g_{k,\delta}'(D_k)-g_{k,\delta}'(\DD )$ in the weak formulations of 
\eqref{3.nk}, \eqref{3.pk}, and \eqref{3.Ndk}, respectively, add the equations, 
and use the Poisson equation \eqref{3.Vk}:
\begin{align*}
   H_{k,\delta}[&(n_k,p_k,D_k,V_k)(t)] - H_{k,\delta}^I
	= \int_0^t\langle\pa_t n_k,g_{k,\delta}'(n_k)-g_{k,\delta}'(\nD )\rangle ds \\ 
	&\phantom{xx}{}+ \int_0^t\langle\pa_t p_k,g_{k,\delta}'(p_k)
	-g_{k,\delta}'(\pD )\rangle ds 
	+ \int_0^t\langle\pa_t D_k,g_{k,\delta}'(D_k)-g_{k,\delta}'(\DD )\rangle ds \\
	&\phantom{xx}{}- \int_0^t\langle\pa_t(n_k-p_k-D_k),V_k-\VD \rangle ds \\
	&= -\frac{1}{\eps}\int_0^t\int_\Omega 
	\na(g_{k,\delta}'(n_k)-g_{k,\delta}'(\nD )
	-V_k+\VD )\cdot(\na n_k-T_k(n_k)\na V_k)dxds \\
	&\phantom{xx}{}- \frac{1}{\eps}\int_0^t\int_\Omega
	\na(g_{k,\delta}'(p_k)-g_{k,\delta}'(\pD )+V_k-\VD )
	\cdot(\na p_k+T_k(p_k)\na V_k)dxds \\
	&\phantom{xx}{}- \int_0^t\int_\Omega \na(g_{k,\delta}'(D_k)+V_k)
	\cdot(\na D_k+T_k(D_k)\na V_k)dxds,
\end{align*}
where $\langle\cdot,\cdot\rangle$ is the duality product between
$H_D^1(\Omega)'$ and $H_D^1(\Omega)$ or between $H^1(\Omega)'$ and 
$H^1(\Omega)$, depending on the context. Since
\begin{align*}
  \na n_k - T_k(n_k)\na V_k &= \sqrt{T_k(n_k)+\delta}
	\big(\na h_{k,\delta}(n_k) - \sqrt{T_k(n_k)+\delta}\na V_k\big) + \delta\na V_k, \\
	\na(g_{k,\delta}'(n_k)-V_k) 
	&= \frac{\na h_{k,\delta}(n_k)
	- \sqrt{T_k(n_k)+\delta}\na V_k}{\sqrt{T_k(n_k)+\delta}},
\end{align*}
we obtain
\begin{align*}
  \na&(g_{k,\delta}'(n_k)-g_{k,\delta}'(\nD )
	-V_k+\VD )\cdot(\na n_k-T_k(n_k)\na V_k) \\
	&= \big|\na h_{k,\delta}(n_k)-\sqrt{T_k(n_k)+\delta}\na V_k\big|^2
	+ \frac{\delta\na V_k}{\sqrt{T_k(n_k)+\delta}}
	\cdot\big(\na h_{k,\delta}(n_k)-\sqrt{T_k(n_k)+\delta}\na V_k\big) \\
	&\phantom{xx}{}- \sqrt{T_k(n_k)+\delta}\na(g_{k,\delta}'(\nD)-\VD)\cdot
	\big(\na h_{k,\delta}(n_k)-\sqrt{T_k(n_k)+\delta}\na V_k\big) \\
	&\phantom{xx}{}- \delta\na(g_{k,\delta}'(\nD)-\VD)\cdot\na V_k \\
	&\ge \frac12\big|\na h_{k,\delta}(n_k)-\sqrt{T_k(n_k)+\delta}\na V_k\big|^2
	- 2(T_k(n_k)+\delta)|\na(g_{k,\delta}'(\nD)-\VD)|^2
	- 2\delta|\na V_k|^2.
\end{align*}
The terms involving $p_k$ and $D_k$ are estimated in a similar way. We infer that
\begin{align}\label{3.aux2}
  H_{k,\delta}[&(n_k,p_k,D_k,V_k)(t)] 
	+ \frac{1}{2\eps}\int_0^t\int_\Omega\big|\na h_{k,\delta}(n_k)
	-\sqrt{T_k(n_k)+\delta}\na V_k\big|^2 dxds \\
	&\phantom{xx}{}+ \frac{1}{2\eps}\int_0^t\int_\Omega\big|\na h_{k,\delta}(p_k)
	+\sqrt{T_k(p_k)+\delta}\na V_k\big|^2 dxds \nonumber \\
	&\phantom{xx}{}+ \frac{1}{2}\int_0^t\int_\Omega\big|\na h_{k,\delta}(D_k)
	+\sqrt{T_k(D_k)+\delta}\na V_k\big|^2 dxds \nonumber \\
	&\le H_{k,\delta}^I + C\Lambda_{\eps,k,\delta}\int_0^t\int_\Omega
	(T_k(n_k)+T_k(p_k)+\delta)dxds 
	+ \delta C\int_0^t\int_\Omega|\na V_k|^2 dxds \nonumber \\
	&\le H_{k,\delta}^I + C\Lambda_{\eps,k,\delta}
	+ C(\Lambda_{\eps,k,\delta}+\delta)\int_0^t H_{k,\delta}
	[(n_k,p_k,D_k,V_k)(s)]ds, \nonumber
\end{align}
using bound \eqref{3.Hk} for $H_{k,\delta}$ and inequality $T_k(s)\le s$ for $s\ge 0$. 
Then, by Gronwall's lemma,
$$
  \sup_{0<t<T}H_{k,\delta}[(n_k,p_k,D_k,V_k)(t)] 
	\le \big(H_{k,\delta}^I + C\Lambda_{\eps,k,\delta}\big) 
	e^{C(\Lambda_{\eps,k,\delta}+\delta)T}.
$$
Using this information in \eqref{3.aux2} then yields \eqref{3.eik1},
and $C(H_{k,\delta}^I,\Lambda_{\eps,k,\delta},T,\delta)=0$ if 
$\Lambda_{\eps,k,\delta}=0$ and $\delta=0$.  
\end{proof}

The next step is the limit $\delta\to 0$ in \eqref{3.eik1}. To this end,
we define
\begin{align}
  \Lambda_{\eps,k} 
	&= \frac{1}{2\eps}\big(\|\na(g_{k}'(\nD )-\VD)\|_{L^\infty(Q_T)}^2
	+ \|\na(g_{k}'(\pD )+\VD )\|_{L^\infty(Q_T)}^2\big), \nonumber \\
  H_k[n,p,D,V] &= \int_\Omega\bigg(G_k(n,\nD)+G_k(p,\pD)+G_k(D,\DD)
  + \frac{\lambda^2}{2}|\na(V-\VD)|^2\bigg)dx, \nonumber \\
  G_k(s,\overline{s}) &= g_k(s)-g_k(\overline{s}) 
  - g_k'(\overline{s})(s-\overline{s}), \nonumber \\
	g_k(s) &= \int_0^s\int_1^y\frac{dz}{T_k(z)}dy, \quad
	h_{k}(s) = \int_0^s\frac{dy}{\sqrt{T_k(y)}}, \quad s\ge 0, \nonumber \\
	H_{k}^I &= H_{k}[n^I,p^I,D^I,V^I]. \nonumber
\end{align}

\begin{lemma}[Regularized free energy inequality II]
Let $(n_k,p_k,D_k,V_k)$ be a weak solution to the approximate problem
\eqref{3.nk}--\eqref{3.nbcd}. Then there exists a constant 
$C(H_{k}^I,\Lambda_{\eps,k},T)>0$ such that for all $0<t<T$,
\begin{align}\label{3.eik}
  &H_{k}[n_k(t),p_k(t),D_k(t),V_k(t)]
	+ \frac{1}{2\eps}\int_0^t\int_\Omega
	|\na h_{k}(n_k)-\sqrt{T_k(n_k)}\na V_k)|^2 dxds \\
	&\phantom{xx}{}
	+ \frac{1}{2\eps}\int_0^t\int_\Omega 
	|\na h_{k}(p_k)+\sqrt{T_k(p_k)}\na V_k)|^2dxds \nonumber \\
	&\phantom{xx}{}
	+ \frac{1}{2}\int_0^t\int_\Omega|\na h_{k}(D_k)+\sqrt{T_k(D_k)}\na V_k)|^2 dxds  
	\le H_k^I + C(H_k^I,\Lambda_{\eps,k},T), \nonumber
\end{align}
and the constant $C(H_k^I,\Lambda_{\eps,k},T)$ vanishes if $\Lambda_{\eps,k}=0$.
\end{lemma}

\begin{proof}
The lemma follows after performing the limit $\delta\to 0$ in \eqref{3.eik1}.
We claim that $\na h_{k,\delta}(n_k)-\sqrt{T_k(n_k)+\delta}\na V_k\rightharpoonup
\na h_k(n_k)-\sqrt{T_k(n_k)}\na V_k$ weakly in $L^2(Q_T)$ as $\delta\to 0$. 
Indeed, we know that
$$
  \big|\sqrt{T_k(n_k)+\delta}-\sqrt{T_k(n_k)}\big|
	= \frac{\delta}{|\sqrt{T_k(n_k)+\delta}+\sqrt{T_k(n_k)}|} 
	\le \sqrt{\delta}\to 0
$$
and, by monotone convergence, $h_{k,\delta}(n_k)\to h_k(n_k)$ a.e.\ in $Q_T$. 
Since $h_k(s)\le C(k)$ for $s\ge 0$, we deduce from
dominated convergence that $h_{k,\delta}(n_k)\to h_k(n_k)$ strongly in
$L^q(Q_T)$ for any $q<\infty$. Finally, 
$\na h_{k,\delta}(n_k)-\sqrt{T_k(n_k)+\delta}\na V_k$ is bounded in $L^2(Q_T)$ 
uniformly in $\delta$, and there exists a subsequence that converges weakly
in $L^2(Q_T)$. The previous arguments show that we can identify the
weak limit, showing the claim. The other terms in \eqref{3.eik1} can be
treated in a similar way. The limit $\delta\to 0$ proves \eqref{3.eik}.
\end{proof}

The free energy inequality \eqref{3.eik} implies some uniform bounds, which are
collected in the following lemma.

\begin{lemma}[Global estimates for the approximate problem]\label{lem.estapprox}
Let $(n_k,p_k,D_k,V_k)$ be a weak solution to the approximate problem
\eqref{3.nk}--\eqref{3.nbcd}. Then there exists a constant $C>0$, which is
independent of $k$ and $\eps$, such that
\begin{align}
  \|g_k(n_k)\|_{L^\infty(0,T;L^1(\Omega))}
	+ \|g_k(p_k)\|_{L^\infty(0,T;L^1(\Omega))}
	+ \|g_k(D_k)\|_{L^\infty(0,T;L^1(\Omega))} &\le C, \label{3.gk} \\
  \|n_k\log n_k\|_{L^\infty(0,T;L^1(\Omega))}
	+ \|p_k\log p_k\|_{L^\infty(0,T;L^1(\Omega))}
	+ \|D_k\log D_k\|_{L^\infty(0,T;L^1(\Omega))} &\le C, \label{3.nlogn} \\
  \big\|\sqrt{T_k(n_k)}\na V_k\big\|_{L^\infty(0,T;L^1(\Omega))} 
	+ \big\|\sqrt{T_k(p_k)}\na V_k\big\|_{L^\infty(0,T;L^1(\Omega))} 
	& \label{3.sqrtTkVk} \\
	{}+ \big\|\sqrt{T_k(D_k)}\na V_k\big\|_{L^\infty(0,T;L^1(\Omega))}&\le C, 
	\nonumber \\
	\|h_k(n_k)\|_{L^2(0,T;W^{1,1}(\Omega))} 
	+ \|h_k(p_k)\|_{L^2(0,T;W^{1,1}(\Omega))} 
	+ \|h_k(D_k)\|_{L^2(0,T;W^{1,1}(\Omega))} &\le C. \label{3.hk}
\end{align}
\end{lemma}

\begin{proof}
Estimate \eqref{3.gk} is a consequence of the free energy inequality \eqref{3.eik},
and \eqref{3.nlogn} follows from \eqref{3.gk} and 
$$
  g_k(s) \ge \int_0^s\int_1^y\frac{dz}{z}dy = s(\log s-1) \ge \frac12 s\log s
$$
for sufficiently large $s>1$. Lemma \ref{lem.aux} in the Appendix shows that
\begin{equation*}
  \big\|\sqrt{T_k(n_k)}\big\|_{L^\infty(0,T;L^2(\Omega))}
	\le C + C\|g_k(n_k)\|_{L^\infty(0,T;L^1(\Omega))}^{1/2} \le C.
\end{equation*}
Then the $L^\infty(0,T;L^2(\Omega))$ bound for $\na V_k$ from the free energy
inequality \eqref{3.eik} implies that
$$
  \big\|\sqrt{T_k(n_k)}\na V_k\big\|_{L^\infty(0,T;L^1(\Omega))}
	\le \big\|\sqrt{T_k(n_k)}\big\|_{L^\infty(0,T;L^2(\Omega))}
	\|\na V_k\|_{L^\infty(0,T;L^2(\Omega))} \le C,
$$
which proves \eqref{3.sqrtTkVk}.
Next, by the bound on the entropy production from \eqref{3.eik},
\begin{align*}
  \|\na h_k(n_k)\|_{L^2(0,T;L^1(\Omega))} 
	&= \big\|\na h_k(n_k)-\sqrt{T_k(n_k)}\na V_k\big\|_{L^2(0,T;L^1(\Omega))} \\
	&\phantom{xx}{}+ \big\|\sqrt{T_k(n_k)}\na V_k\big\|_{L^2(0,T;L^1(\Omega))} \le C.
\end{align*}
Finally, we deduce from the proof of Lemma \ref{lem.aux} in the Appendix that
\begin{equation}\label{3.esthk}
  \|h_k(n_k)\|_{L^\infty(0,T;L^2(\Omega))} 
	\le C+C\|g_k(n_k)\|_{L^\infty(0,T;L^1(\Omega))}^{1/2}\le C
\end{equation}
such that \eqref{3.hk} follows. Similar bounds hold for $p_k$ and $D_k$.
\end{proof}

The estimates of the previous lemma are not sufficient to show that the
current density
$$
  \na n_k - T_k(n_k)\na V_k = \sqrt{T_k(n_k)}
	\big(\na h_k(n_k)-\sqrt{T_k(n_k)}\na V_k\big)
$$
is uniformly bounded. Therefore, we prove stronger estimates
in $\Omega_\delta:=\{x\in\Omega:\operatorname{dist}(x,\pa\Omega)>\delta\}$,
which allow us to apply the Aubin--Lions lemma.

\begin{lemma}[Local estimates for the approximate problem]
Let $(n_k,p_k,D_k,V_k)$ be a weak solution to the approximate problem
\eqref{3.nk}--\eqref{3.nbcd} and let $r=(2+2d)/(1+2d)$, $r'=2d+2$. 
Then there exists a constant $C(\delta)>0$, 
depending on $\delta$ but not on $k$ or $\eps$, such that
\begin{align}
  \|n_k\|_{L^r(0,T;W^{1,r}(\Omega_\delta))}
	+ \|p_k\|_{L^r(0,T;W^{1,r}(\Omega_\delta))}
	+ \|D_k\|_{L^r(0,T;W^{1,r}(\Omega_\delta))} &\le C(\delta), \label{3.W1r} \\
	\|\pa_t n_k\|_{L^r(0,T;W^{-1,r}(\Omega_\delta))}
	+ \|\pa_t p_k\|_{L^r(0,T;W^{-1,r}(\Omega_\delta))}
	+ \|\pa_t D_k\|_{L^r(0,T;W^{1,r'}(\Omega_\delta)')} &\le C(\delta). \label{3.time}
\end{align}
\end{lemma}

\begin{proof}
We define the cutoff function $\xi_\delta\in C_0^1(\R^d)$ such that
$0\le\xi_\delta\le 1$ in $\R^d$, $\xi_\delta=1$ in $\Omega_\delta$,
$\xi_\delta=0$ in $\Omega\setminus\Omega_{\delta/2}$, and 
$\|\na\xi_\delta\|_{L^\infty(\R^d)}\le C_\xi/\delta$. The bound for the
entropy production in \eqref{3.eik} and the property
$\na h_k(n_k)=T_k(n_k)^{-1/2}\na n_k$ imply that
\begin{align*}
  \int_0^T&\int_\Omega\big(|\na h_k(n_k)|^2 + T_k(n_k)|\na V_k|^2\big)\xi_\delta^2 dxdt
	= \int_0^T\big|\na h_k(n_k)-\sqrt{T_k(n_k)}\na V_k\big|^2\xi_\delta^2 dxdt \\
	&{}+ 2\int_0^T\int_\Omega \na n_k\cdot\na V_k\xi_\delta^2 dxdt
	\le C + 2\int_0^T\int_\Omega \na n_k\cdot\na V_k\xi_\delta^2 dxdt.
\end{align*}
Similar computations for $p_k$ and $D_k$ lead to
\begin{align}\label{3.nah}
  \int_0^T&\int_\Omega\big(|\na h_k(n_k)|^2 + |\na h_k(p_k)|^2 + |\na h_k(D_k)|^2\big)
	\xi_\delta^2 dxdt \\
	&\phantom{xx}{}+ \int_0^T\int_\Omega\big(T_k(n_k)+T_k(p_k)+T_k(D_k)\big)
	|\na V_k|^2\xi_\delta^2 dxdt \nonumber \\
	&\le C + 2\int_0^T\int_\Omega\na(n_k-p_k-D_k)\cdot\na V_k\xi_\delta^2 dxdt.
	\nonumber
\end{align}
By the Poisson equation \eqref{3.Vk} and Young's inequality, we find for the
last integral that
\begin{align*}
  \int_0^T&\int_\Omega\na(n_k-p_k-D_k)\cdot\na V_k\xi_\delta^2 dxdt \\
	&= -\frac{1}{\lambda^2}\int_0^T\int_\Omega(n_k-p_k-D_k)(n_k-p_k-D_k+A)\xi_\delta^2
	dxdt \\
	&\phantom{xx}{}- 2\int_0^T\int_\Omega(n_k-p_k-D_k)\xi_\delta \na V_k\cdot\na\xi_\delta
	dxdt \\
	&\le -\frac{1}{2\lambda^2}\int_0^T\int_\Omega(n_k-p_k-D_k)^2\xi_\delta^2 dxdt
	+ \frac{1}{\lambda^2}\int_0^T\int_\Omega A^2\xi_\delta^2 dxdt \\
	&\phantom{xx}{}+ 4\lambda^2\int_0^T\int_\Omega|\na V_k|^2|\na\xi_\delta|^2 dxdt.
\end{align*}
The free energy inequality \eqref{3.eik} shows that $\na V_k$ is uniformly bounded
in $L^2(Q_T)$. Therefore, using $|\na\xi_\delta|^2\le C_\xi^2\delta^{-2}$,
\eqref{3.nah} becomes
\begin{align}\label{3.nah2}
   \int_0^T&\int_\Omega\big(|\na h_k(n_k)|^2 + |\na h_k(p_k)|^2 + |\na h_k(D_k)|^2\big)
	\xi_\delta^2 dxdt \\
	&\phantom{xx}{}+ \int_0^T\int_\Omega\big(T_k(n_k)+T_k(p_k)+T_k(D_k)\big)
	|\na V_k|^2\xi_\delta^2 dxdt \nonumber \\
	&\phantom{xx}{}+ \frac{1}{2\lambda^2}\int_0^T\int_\Omega(n_k-p_k-D_k)^2
	\xi_\delta^2 dxdt	\le C + C\delta^{-2}. \nonumber 
\end{align}
This leads, together with \eqref{3.esthk}, to the bound
$$
  \|h_k(n_k)\|_{L^2(0,T;H^1(\Omega_\delta))} 
	+ \big\|\sqrt{T_k(n_k)}\na V_k\big\|_{L^2(0,T;L^2(\Omega_\delta))} \le C\delta^{-1},
$$
and similarly for $p_k$ and $D_k$. 

Next, we use the Gagliardo--Nirenberg
inequality with $q=2+2/d$ \cite[p.~95]{Jue16} and \eqref{3.esthk}:
$$
  \|h_k(n_k)\|_{L^q(0,T;L^q(\Omega_\delta))}
	\le C\|h_k(n_k)\|_{L^2(0,T;H^1(\Omega_\delta))}^{d/(d+1)}
	\|h_k(n_k)\|_{L^\infty(0,T;L^1(\Omega_\delta))}^{1/(d+1)}
  \le C\delta^{-d/(d+1)}.
$$
We deduce from Lemma \ref{lem.aux} in the Appendix that
\begin{equation}\label{3.TkLq}
  \big\|\sqrt{T_k(n_k)}\big\|_{L^q(\Omega_\delta\times(0,T))}
	\le C\|h_k(n_k)\|_{L^{q}(\Omega_\delta\times(0,T))} \le C(\delta).
\end{equation}
It follows from these estimates and H\"older's inequality that
\begin{align*}
  \|\na n_k\|_{L^r(0,T;L^r(\Omega_\delta))}
	&= \big\|\sqrt{T_k(n_k)}\na h_k(n_k)\big\|_{L^r(0,T;L^r(\Omega_\delta))} \\
	&\le \big\|\sqrt{T_k(n_k)}\big\|_{L^q(0,T;L^q(\Omega_\delta))}
	\|\na h_k(n_k)\|_{L^2(0,T;L^2(\Omega_\delta))} \le C(\delta),
\end{align*}
recalling that $r=(2+2d)/(1+2d)>1$. Similar estimates are derived for
$\na p_k$ and $\na D_k$. Thanks to the Poincar\'e--Wirtinger
inequality and \eqref{3.nlogn}, this shows \eqref{3.W1r}.
Because of the $L^q(\Omega_\delta\times(0,T))$ bound for 
$\sqrt{T_k(n_k)}$ from \eqref{3.TkLq} and the $L^2(\Omega_\delta\times(0,T))$ 
bound for $\sqrt{T_k(n_k)}\na V_k$ from \eqref{3.nah2}, 
$$
  \na n_k - T_k(n_k)\na V_k = \na n_k - \sqrt{T_k(n_k)}\cdot\sqrt{T_k(n_k)}\na V_k
$$
is uniformly bounded in $L^r(\Omega_\delta\times(0,T))$ (depending on $\delta$).
Consequently, $\pa_t n_k$ is uniformly bounded in 
$L^r(0,T;W^{-1,r}(\Omega_\delta))$.
The uniform bounds for $p_k$ and $D_k$ are proved in an analogous way.
\end{proof}

The proof shows that the current density $\na n_k-T_k(n_k)\na V_k$ (and similar for
$p_k$ and $D_k$) is bounded in $L^r(\Omega_\delta\times(0,T))$ uniformly in $k$.
This improves the estimates of Lemma \ref{lem.estapprox}.

\subsection{The limit $k\to\infty$}\label{sec.limitk}

Thanks to estimates \eqref{3.W1r} and \eqref{3.time}, the Aubin--Lions lemma
implies, for any fixed $\delta>0$, the existence of a subsequence of
$(n_k,p_k,D_k)$, which is not relabeled, such that
\begin{equation*}
  n_k\to n,\ p_k\to p,\ D_k\to D\quad\mbox{strongly in }
	L^r(\Omega_\delta\times(0,T))\mbox{ as }k\to\infty.
\end{equation*}
By the Theorem of De la Vall\'ee--Poussin, applied to \eqref{3.nlogn},
the limit functions are uniquely determined in $Q_T$ by the weak
convergence of $(n_k,p_k,D_k)$ in $L^1(Q_T)$. We choose $\delta=1/m$ for
$m\in\N$, $m\ge 1$ and apply a Cantor diagonal argument to deduce the existence of
$\delta$-independent subsequences of $(n_k,p_k,D_k)$, which are strongly
converging to $(n,p,D)$ in $L^s(\Omega_\delta\times(0,T))$ for $1<s<r$ and
every $\delta=1/m$ and consequently also for any $0<\delta<1$, since
$\Omega_\delta\subset\Omega_{\delta'}$ for $\delta>\delta'$. This convergence
and the weak convergence $n_k\rightharpoonup n$ in $L^1(Q_T)$ as $k\to\infty$
imply that
\begin{equation}\label{3.equi}
  \limsup_{k\to\infty}\int_0^T\int_\Omega|n_k-n|dxdt
	\le \limsup_{k\to\infty}\int_0^T\int_{\Omega\setminus\Omega_\delta}|n_k-n|dxdt
	\le 2\sup_{k\in\N}\int_0^T\int_{\Omega\setminus\Omega_\delta}n_k dxdt.
\end{equation}
By the Theorem of De la Vall\'ee--Poussin again, estimate \eqref{3.nlogn}
implies the uniform integrability of $(n_k)_{k\in\N}$, such that we conclude
from \eqref{3.equi} that
$$
  \limsup_{k\to\infty}\int_0^T\int_\Omega|n_k-n|dxdt\le C(\delta)\to 0
	\quad\mbox{as }\delta\to 0.
$$
This means that 
$$
  n_k\to n,\ p_k\to p,\ D_k\to D\quad\mbox{strongly in }
	L^1(Q_T).
$$
We claim that this convergence implies that $T_k(n_k)\to n$ strongly in
$L^1(Q_T)$ and similarly for $p_k$ and $D_k$. Indeed, we infer from
bound \eqref{3.nlogn} that, as $k\to\infty$,
\begin{align*}
  \int_0^T\int_\Omega|T_k(n_k)-n_k|dxdt 
	&\le \int_0^T\int_{\{n_k\ge k\}}|k-n_k| dxdt \\
	&\le \int_0^T\int_{\{n_k\ge k\}}n_k\frac{\log n_k}{\log k}dxdt
	\le \frac{C}{\log k}\to 0.
\end{align*}
Then the convergence $n_k\to n$ strongly in $L^1(Q_T)$ shows the claim.

Now, the limit $k\to\infty$ in the approximate equations is rather standard
except the limit in the flux term. For this, we observe that the bound on the
entropy production in \eqref{3.eik} yields, possibly for a subsequence, that
for $k\to\infty$,
\begin{equation}\label{3.aux}
  \na h_k(n_k) - \sqrt{T_k(n_k)}\na V_k \rightharpoonup\xi
	\quad\mbox{weakly in }L^2(Q_T).
\end{equation}
We wish to identify $\xi$. For this, we claim that $h_k(n_k)-2\sqrt{n_k}\to 0$ 
in $L^2(Q_T)$. An elementary computation shows that
$h_k(s)=2\sqrt{s}$ for $0\le s\le k$ and $h_k(s)=s/\sqrt{k}+\sqrt{k}$ for $s\ge k$.
Therefore, 
$$
  \sup_{0<t<T}\int_\Omega|h_k(n_k)-2\sqrt{n_k}| dx
	= \frac{1}{\sqrt{k}}\sup_{0<t<T}
    \int_{\{n_k>k\}}\big(\sqrt{n_k}-\sqrt{k}\big)^2 dx
	\le \frac{C}{\sqrt{k}}\to 0,
$$
where the constant $C>0$ depends on the $L^\infty(0,T;L^1(\Omega))$ norm of $n_k$.
We infer from $\sqrt{n_k}\to\sqrt{n}$ strongly in $L^2(Q_T)$ that
$h_k(n_k)\to 2\sqrt{n}$ strongly in $L^1(Q_T)$ and consequently,
$$
  \|\na(h_k(n_k)-2\sqrt{n})\|_{L^2(0,T;W^{1,\infty}(\Omega)')}
	\le \|h_k(n_k)-2\sqrt{n}\|_{L^\infty(0,T;L^1(\Omega))}\to 0
$$
or $\na h_k(n_k)\to 2\na\sqrt{n}$ strongly in $L^2(0,T;W^{1,\infty}(\Omega)')$.
The free energy inequality \eqref{3.eik} implies, possibly for a subsequence,
that $\na V_k\rightharpoonup\na V$ weakly* in $L^\infty(0,T;L^2(\Omega))$. 
The limit $\sqrt{T_k(n_k)}\to \sqrt{n}$ strongly in $L^2(Q_T)$ leads to
$\sqrt{T_k(n_k)}\na V_k\rightharpoonup \sqrt{n}\na V$ weakly in $L^1(Q_T)$.
These convergences imply that $\xi=2\na\sqrt{n}-\sqrt{n}\na V$ and,
using \eqref{3.aux} and $\sqrt{n_k}\to\sqrt{n}$ in $L^2(Q_T)$ again,
$$
  \na n_k - n_k\na V_k = \sqrt{n_k}\big(2\na\sqrt{n_k}-\sqrt{n_k}\na V_k\big)
	\rightharpoonup \na n-n\na V\quad\mbox{weakly in }L^1(Q_T).
$$
This estimate shows that for all $\chi\in L^\infty(0,T)$ and
$\phi\in H^s(\Omega)\cap H_D^1(\Omega)$ with $s>1+d/2$,
\begin{align*}
  \eps\int_0^T\chi\langle\pa_t n_k,\phi\rangle dt
	&= -\int_0^T\int_\Omega\chi(\na n_k-n_k\na V_k)\cdot\na\phi dxdt \\
	&\to -\int_0^T\int_\Omega\chi(\na n-n\na V)\cdot\na\phi dxdt,
\end{align*}
since $H^s(\Omega)\hookrightarrow W^{1,\infty}(\Omega)$. The space
$X=H^s(\Omega)\cap H_D^1(\Omega)$ is reflexive and so does its dual.
Thus, we can apply \cite[Lemma 6]{ChJu06} to conclude that
$\pa_t n_k\rightharpoonup w$ weakly in $L^1(0,T;X')$ for some $w$. 
We can identify $w=\pa_t n$ since $n_k\to n$ strongly in $L^1(Q_T)$ 
and so, $\pa_t n_k\rightharpoonup \pa_t n$ in the sense of distributions.
Then the limit $k\to\infty$ in the weak formulation
$$
  \eps\int_0^T\langle \pa_t n_k,\phi\rangle dt
	+ \int_0^T\int_\Omega(\na n_k-n_k\na V_k)\cdot\na\phi dxdt = 0
$$
leads to
$$
  \eps\int_0^T\langle\pa_t n,\phi\rangle dt
	+ \int_0^T\int_\Omega(\na n-n\na V)\cdot\na\phi dxdt = 0
$$
for all $\phi\in L^\infty(0,T;X)$.
The limit $k\to\infty$ for $p_k$ and $D_k$ is performed in a similar way.


\section{Proof of Theorem \ref{thm.infty}}\label{sec.infty}

We show that a weak solution to \eqref{1.n}--\eqref{1.nbcd} is bounded in the
case of two space dimensions. First, we prove an $L^\infty(0,T;L^2(\Omega))$ bound.

\begin{lemma}\label{lem.L2}
Let $d\le 2$. Then there exists $C>0$, depending on the $L^\infty(0,T;L^1(\Omega))$
bounds of $n_k\log n_k$, $p_k\log p_k$, and $D_k\log D_k$ but independent of
$k$ and $\eps$, such that
\begin{equation*}
  \sqrt{\eps}\|n_k\|_{L^\infty(0,T;L^2(\Omega))}
	+ \sqrt{\eps}\|p_k\|_{L^\infty(0,T;L^2(\Omega))}
	+ \|D_k\|_{L^\infty(0,T;L^2(\Omega))} \le C.
\end{equation*}
\end{lemma}

\begin{proof}
We use the test function $D_k$ in the weak formulation of \eqref{3.Ndk}, the
inequality $T_k(D_k)\le D_k$, and apply H\"older's inequality:
\begin{align}
  \frac12\int_\Omega &(D_k(t)^2-(D^I)^2) dx + \int_0^t\int_\Omega|\na D_k|^2 dxds
	= - \int_0^t\int_\Omega T_k(D_k)\na V_k\cdot\na D_k dxds \label{5.aux} \\
	&\le \int_0^t\|D_k\|_{L^{2r_0/(r_0-2)}(\Omega)}\|\na V_k\|_{L^{r_0}(\Omega)}
	\|\na D_k\|_{L^{2}(\Omega)}ds, \nonumber
\end{align}
where $r_0>2$ is from Lemma \ref{lem.vinfty} in the Appendix. 
The second term on the left-hand side is estimated by using the
Poincar\'e--Wirtinger inequality:
\begin{align*}
  \int_0^t\int_\Omega|\na D_k|^2 dxds &\ge C\int_0^t\|D_k\|_{H^1(\Omega)}^2 ds 
	- C\int_0^t\|D_k\|_{L^1(\Omega)}^2ds \\
	&\ge C_1\int_0^t\|D_k\|_{H^1(\Omega)}^2 ds - C_2,
\end{align*}
where $C_2>0$ depends on $T$ and the $L^\infty(0,T;L^1(\Omega))$ norm of $D_k$.
For the right-hand side of \eqref{5.aux}, we use 
Lemma \ref{lem.GN} with $q=2r_0/(r_0+2)$ and Lemma \ref{lem.vinfty}:
\begin{align*}
	\|D_k\|_{L^{2r_0/(r_0+2)}(\Omega)}
	&\le \delta\|D_k\|_{H^1(\Omega)}^{(r_0-2)/(2r_0)}
	\|D_k\log D_k\|_{L^1(\Omega)}^{(r_0+2)/(2r_0)} + C(\delta)\|D_k\|_{L^1(\Omega)} \\
	&\le \delta C\|D_k\|_{H^1(\Omega)}^{(r_0-2)/(2r_0)} + C(\delta), \\
	\|\na V_k\|_{L^{r_0}(\Omega)} 
	&\le C(1+\|n_k-p_k-D_k+A\|_{L^{2r_0/(r_0+2)}(\Omega)}) \\
	&\le C\big(1 + \|n_k\|_{H^1(\Omega)}^{(r_0-2)/(2r_0)}
	+ \|p_k\|_{H^1(\Omega)}^{(r_0-2)/(2r_0)} 
	+ \|D_k\|_{H^1(\Omega)}^{(r_0-2)/(2r_0)}\big),
\end{align*}
where $C>0$ and $C(\delta)>0$ depend on the $L^\infty(0,T;L^1(\Omega))$ norms of 
$n_k\log n_k$, $p_k\log p_k$, and $D_k\log D_k$. 
We conclude from \eqref{5.aux} that
\begin{align*}
  \|&D_k(t)\|_{L^2(\Omega)}^2 + C\int_0^t\|D_k\|_{H^1(\Omega)}^2 ds
	\le \|D^I\|_{L^2(\Omega)}^2 + C(\delta) \\
	&\phantom{xx}{}{}+ \delta C\int_0^t\|D_k\|_{H^1(\Omega)}^{1+(r_0-2)/(2r_0)}
	C\big(1 + \|n_k\|_{H^1(\Omega)}^{(r_0-2)/(2r_0)}
	+ \|p_k\|_{H^1(\Omega)}^{(r_0-2)/(2r_0)} 
	+ \|D_k\|_{H^1(\Omega)}^{(r_0-2)/(2r_0)}\big)ds \\
	&\le \|D^I\|_{L^2(\Omega)}^2 + C(\delta)
	+ \delta C\int_0^t\big(1 + \|n_k\|_{H^1(\Omega)}^2 + \|p_k\|_{H^1(\Omega)}^2
	+ \|D_k\|_{H^1(\Omega)}^2\big)ds.
\end{align*}
We can apply Young's inequality in the last step since
$1+(r_0-2)/(2r_0)=3/2-1/r_0<2$.
Similar inequalities can be derived for $n_k$ and $p_k$ (using the test functions
$n_k-\nD$ and $p_k-\pD$). Adding these inequalities and choosing $\delta>0$
sufficiently small leads to
\begin{align*}
  \eps\|&n_k(t)\|_{L^2(\Omega)}^2 + \eps\|p_k(t)\|^{2}_{L^2(\Omega)} 
	+ \|D_k(t)\|^{2}_{L^2(\Omega)} \\
	&{}+ C\int_0^t\big(\|n_k\|_{H^1(\Omega)}^2+\|p_k\|_{H^1(\Omega)}^2
	+\|D_k\|_{H^1(\Omega)}^2\big) ds	\le C,
\end{align*}
and the constant $C>0$ depends on the initial data and the $L^\infty(0,T;L^1(\Omega))$
norms of $n_k\log n_k$, $p_k\log p_k$, and $D_k\log D_k$. 
\end{proof}

\begin{lemma}\label{lem.Vk}
Let $d\le 2$. Then there exists $C(\eps)>0$, independent of $k$, such that
$$
  \|V_k\|_{L^\infty(0,T;W^{1,r_0}(\Omega))} \le C(\eps),
$$
where $r_0>2$ is from Lemma \ref{lem.vinfty} in the Appendix.
\end{lemma}

\begin{proof}
The free energy inequality implies that $V_k$ is uniformly bounded in
$L^\infty(0,T;H^1(\Omega))$. Then the estimate for $V_k$ is a consequence of 
Lemma \ref{lem.vinfty} and Lemma \ref{lem.L2}:
$$
  \|V_k\|_{L^\infty(0,T;W^{1,r_0}(\Omega))}
	\le C\big(\|n_k-p_k-D_k+A\|_{L^\infty(0,T;L^{2r_0/(r_0+2)}(\Omega))}+1\big)
	\le C(\eps),
$$
since $2r_0/(r_0+2)<2$.
\end{proof}

The following lemma provides $L^\infty$ bounds depending on the truncation
parameter $k$. This result is used to prove uniform $L^\infty$ bounds later.

\begin{lemma}\label{lem.inftyk}
Let $d\le 2$  and $n^I$, $p^I$, $D^I\in L^\infty(\Omega)$. 
Then there exists $C>0$, depending on the $L^\infty(0,T;L^1(\Omega))$
bounds of $n_k\log n_k$, $p_k\log p_k$, $D_k\log D_k$ and on $\eps$, $k$
(and possibly on $T$), such that
$$
  \|n_k\|_{L^\infty(Q_T)} + \|p_k\|_{L^\infty(Q_T)}
	+ \|D_k\|_{L^\infty(Q_T)} \le C(\eps,k).
$$
\end{lemma}

\begin{proof}
Let $q\ge 2$ and $\|D^I\|_{L^\infty(\Omega)}<M<L$.
We set $[z]_L=\min\{L,z\}$, $z_+=\max\{0,z\}$, 
and $\phi_L(z)=([z]_L-M)_+$ for $z\in\R$. Then
$$
  \int_0^z\phi_L(s)^{q-1}ds \ge \frac{1}{q}\phi_L(z)^q\quad\mbox{for }z\ge 0.
$$
Because of the truncation, $\phi_L(D_k)^{q-1}$ is an admissible test function 
in the weak formulation of \eqref{3.Ndk}. Observing that the 
definition of $M$ shows that
$$
  \int_0^t\langle\pa_t D_k,\phi_L(D_k)^{q-1}\rangle ds 
	\ge\frac{1}{q}\int_\Omega \phi_L(D_k(t))^q dx,
$$
we obtain from \eqref{3.Ndk}, H\"older's inequality, and $T_k(D_k)\le k$:
\begin{align}\label{5.aux1}
  \frac{1}{q}&\int_\Omega \phi_L(D_k(t))^q dx 
	+ \frac{4}{q^2}(q-1)\int_0^t \int_\Omega |\na\phi_L(D_k)^{q/2}|^2 dxds \\
	&= -\int_0^t\int_\Omega T_k(D_k)\na V_k\cdot\na\phi_L(D_k)^{q-1}dxds \nonumber \\
	&\le Ck\int_0^t\|\na V_k\|_{L^{r_0}(\Omega)}\|\na\phi_L(D_k)^{q/2}\|_{L^2(\Omega)}
	\|\phi_L(D_k)^{q/2-1}\|_{L^{r'_0}(\Omega)}ds, \nonumber
\end{align}
where $r_0>2$ is from Lemma \ref{lem.vinfty} and $r_0'=2r_0/(r_0-2)>2$.
By definition of the $H^1(\Omega)$ norm,
\begin{equation*}
  \int_0^t\int_\Omega|\na\phi_L(D_k)^{q/2}|^2 dxds
	\ge \int_0^t\big(\|\phi_L(D_k)^{q/2}\|_{H^1(\Omega)}^2 
	- \|\phi_L(D_k)^{q/2}\|_{L^2(\Omega)}^2\big)ds.
\end{equation*}
By Lemma \ref{lem.Vk}, the $L^\infty(0,T;L^{r_0}(\Omega))$ norm of 
$\na V_k$ is bounded uniformly in $k$.
Then, by the Gagliardo--Nirenberg inequality \eqref{GN}, setting $s=(1-2/q)r_0'$:
\begin{align*}
  \|\phi_L(D_k)^{q/2-1}\|_{L^{r_0'}(\Omega)}
	&= \|\phi_L(D_k)^{q/2}\|_{L^s(\Omega)}^{1-2/q} \\
	&\le C\|\phi_L(D_k)^{q/2}\|_{H^1(\Omega)}^{(1-1/s)(1-2/q)}
	\|\phi_L(D_k)^{q/2}\|_{L^1(\Omega)}^{(1/s)(1-2/q)} \\
	&\le C + C\|\phi_L(D_k)^{q/2}\|_{H^1(\Omega)}^{1-1/s}
	\|\phi_L(D_k)^{q/2}\|_{L^1(\Omega)}^{1/s}.
\end{align*}
Inserting these estimates into \eqref{5.aux1} and using Young's inequality
for an arbitrary $\delta>0$, we arrive at
\begin{align*}
  \|\phi_L&(D_k(t))\|_{L^q(\Omega)}^q 
	+ C\int_0^t\|\phi_L(D_k)^{q/2}\|_{H^1(\Omega)}^2 ds \\
	&\le Ckq + C(k)q\int_0^t\|\phi_L(D_k)^{q/2}\|_{L^2(\Omega)}^2 ds
	+ Ckq\int_0^t\|\phi_L(D_k)^{q/2}\|_{H^1(\Omega)}^{2-1/s}
	\|\phi_L(D_k)^{q/2}\|_{L^1(\Omega)}^{1/s}ds \\
	&\le Ckq + C(\delta,k)q^{\max\{1,2s\}}\int_0^t\|\phi_L(D_k)^{q/2}\|_{L^2(\Omega)}^2 ds
	+ \delta\int_0^t\|\phi_L(D_k)^{q/2}\|_{H^1(\Omega)}^2 ds.
\end{align*}
It remains to choose a sufficiently small $\delta>0$ to absorb the 
last term on the right-hand side and to apply Lemma \ref{lem.infty},
which yields
$$
  \|\phi_L(D_k(t))\|_{L^\infty(\Omega)}\le C,\quad t\in(0,T),
$$
where $C>0$ does not depend on $L$. The limit $L\to\infty$ then shows that
$(D_k(t)-M)_+\le C$ and consequently, $D_k(t)\le C+M$ in $Q_T$.
The $L^\infty$ bounds for $n_k$ and $p_k$ are proved in an analogous way by
choosing $M>\max\{\|\nD\|_{L^\infty(\Gamma_D)},\|n^I\|_{L^\infty(\Omega)}\}$
and $M>\max\{\|\pD\|_{L^\infty(\Gamma_D)},\|p^I\|_{L^\infty(\Omega)}\}$,
respectively.
\end{proof}

We proceed with the proof of Theorem \ref{thm.infty}, which is 
technically similar to the previous proof.
Let $q\ge 2$ and $M>\|D^I\|_{L^\infty(\Omega)}$.
We set $\phi(z)=(z-M)_+$ for $z\ge 0$. Lemma \ref{lem.inftyk} guarantees that
$e^t\phi(D_k)^{q-1}$ is an admissible test function in the weak formulation of
\eqref{3.Ndk}. (The factor $e^t$ allows us to derive time-uniform bounds.)
Using $T_k(D_k)\le (D_k-M)_+ + M =\phi(D_k)+M$ and computing 
similarly as in the proof of Lemma \ref{lem.inftyk}, we find that
\begin{align*}
  \|e^t&\phi(D_k(t))\|_{L^q(\Omega)}^q 
	+ C\int_0^t e^s\|\phi(D_k)^{q/2}\|_{H^1(\Omega)}^2 ds
	- C\int_0^t e^s\|\phi(D_k)^{q/2}\|_{L^2(\Omega)}^2 ds \\
	&\le Cq\int_0^t e^s\|\na V_k\|_{L^{r_0}(\Omega)}\|\na\phi(D_k)^{q/2}\|_{L^2(\Omega)}
	\|\phi(D_k)^{q/2-1}(\phi(D_k)+M)\|_{L^{r'_0}(\Omega)}ds.
\end{align*}
recalling that $r'_0=2r_0/(r_0-2)>2$. 
Taking into account the $L^\infty(0,T;W^{1,r_0}(\Omega)$
bound for $V_k$, independent of $k$, and the Gagliardo--Nirenberg
inequality, we compute
\begin{align*}
  \|e^t&\phi(D_k(t))\|_{L^q(\Omega)}^q 
	+ C\int_0^t e^s\|\phi(D_k)^{q/2}\|_{H^1(\Omega)}^2 ds 
	- C\int_0^t e^s\|\phi(D_k)^{q/2}\|_{L^2(\Omega)}^2 ds \\
	&\le Cq\int_0^t e^s\|\na\phi(D_k)^{q/2}\|_{L^2(\Omega)}
	\big(\|\phi(D_k)^{q/2}\|_{L^{r'_0}(\Omega)}
	+ M\|\phi(D_k)^{q/2-1}\|_{L^{r'_0}(\Omega)}\big)ds \\
	&\le Cq\int_0^t e^s\|\phi(D_k)^{q/2}\|_{H^1(\Omega)}
	\Big(\|\phi(D_k)^{q/2}\|_{H^1(\Omega)}^{1-1/r'_0}
	\|\phi(D_k)^{q/2}\|_{L^1(\Omega)}^{1/r'_0} \\
	&\phantom{xx}{}
	+ MC\big(1+\|\phi(D_k)^{q/2}\|_{H^1(\Omega)}^{1-1/s}
    \|\phi(D_k)^{q/2}\|_{L^1(\Omega)}^{1/s}\big)\Big)ds,
\end{align*}
where $s=(1-2/q)r'_0$. Then it follows from Young's inequality for an arbitrary
$\delta>0$ that
\begin{align*}
  \|e^t&\phi(D_k(t))\|_{L^q(\Omega)}^q 
	+ C\int_0^t e^s\|\phi(D_k)^{q/2}\|_{H^1(\Omega)}^2 ds 
	- C\int_0^t e^s\|\phi(D_k)^{q/2}\|_{L^2(\Omega)}^2 ds \\
	&\le Cqe^t + C\delta\int_0^t e^s\|\phi(D_k)^{q/2}\|_{H^1(\Omega)}^2 ds
	+ C(\delta)q^{\max\{1,2r'_0\}}\int_0^t e^s\|\phi(D_k)^{q/2}\|_{L^2(\Omega)}^2 ds.
\end{align*}
Choosing $\delta>0$ sufficiently small, the second term on the right-hand side
is absorbed from the left-hand side, and Lemma \ref{lem.infty} implies that
$$
  \|\phi(D_k(t))\|_{L^\infty(\Omega)} \le C, \quad t>0,
$$
where $C>0$ is independent of $k$ and $T$ (but depending on $\eps$). 
This shows that $D_k(t)\le C+M$
in $\Omega$, $t>0$. The $L^\infty$ bounds for $n_k$ and $p_k$ are proved
in a similar way.


\section{Proof of Theorem \ref{thm.lim}}\label{sec.lim}

We start with the proof of some estimates.

\subsection{A priori estimates}

The free energy of the limit problem is defined as
\begin{align*}
  H_0[D_0,V_0] &= \int_\Omega\bigg(D_0\log\frac{D_0}{\DD } - D_0
	+ \nD  f_0(V_0-\VD ) + \pD  f_0(\VD -V_0)
	+ \frac{\lambda^2}{2}|\na(V_0-\VD )|^2\bigg)dx,
\end{align*}
where $\DD:=\exp(-\VD)$ and the function $f_0$ is given by 
$f_0(s)=(s-1)e^s+1$ for $s\in\R$.

\begin{lemma}[Free energy inequality for the limit problem]\label{lem.eilimit}
Let $(D_0,V_0)$ be a smooth solution to \eqref{1.limD}--\eqref{1.limbcV}.
Then there exists a constant $C>0$, only depending on $H_0[D^I,V^I]$ and $T$, 
such that
$$
  H_0[D_0(t),V_0(t)] + \frac12\int_0^t\int_\Omega D_0|\na(\log D_0+V_0)|^2 dxds
	\le C, \quad 0<t<T.
$$
\end{lemma}

\begin{proof}
We calculate the time derivative of the free energy, using the definitions
$n_0=\nD \exp(V_0-\VD )$, $p_0=\pD \exp(\VD -V_0)$:
\begin{align*}
  \frac{dH_0}{dt}[D_0,V_0] &= \int_\Omega\bigg(\pa_t D_0\log\frac{D_0}{\DD }
	+ (V_0-\VD )\pa_t\big(\nD  e^{V_0-\VD } - \pD  e^{\VD -V_0}\big)
	+ \lambda^2\na(V_0-\VD)\cdot\na\pa_t V_0\bigg)dx \\
	&= \int_\Omega\bigg(\pa_t D_0\log\frac{D_0}{\DD }
	+ (V_0-\VD )\pa_t\big(\nD e^{V_0-\VD } - \pD  e^{\VD -V_0}\big) \\
	&\phantom{xx}{}- (V_0-\VD )\pa_t(n_0-p_0-D_0+A(x))\bigg)dx \\
	&= \int_\Omega\bigg(\log\frac{D_0}{\DD }+V_0-\VD \bigg)\pa_t D_0 dx.
\end{align*}
Inserting the equation for $D_0$ and integrating by parts gives
\begin{align*}
  \frac{dH_0}{dt}[D_0,V_0] &= -\int_\Omega D_0\na(\log D_0+V_0)\cdot\na\big(
	(\log D_0+V_0)-(\log \DD +\VD )\big)dx \\
	&\le -\frac12\int_\Omega D_0|\na(\log D_0+V_0)|^2 dx
	+ \frac12\int_\Omega D_0|\na(\log\DD+\VD)|^2 dx. 
\end{align*}
The last term can be estimated from above by $CH_0[D_0,V_0]$. Then Gronwall's
lemma completes the proof.
\end{proof}

The free energy inequality yields the following uniform bounds:
\begin{align*}
  \|D_0\log D_0\|_{L^\infty(0,T;L^1(\Omega))}
	+ \|V_0\|_{L^\infty(0,T;H^1(\Omega))} &\le C, \\
	\|V_0\exp|V_0|\|_{L^\infty(0,T;L^1(\Omega))} 
	+ \big\|2\na\sqrt{D_0}+\sqrt{D_0}\na V_0\big\|_{L^2(Q_T)} &\le C,
\end{align*}
since $2\na\sqrt{D_0}+\sqrt{D_0}\na V_0 = \sqrt{D_0}\na(\log D_0+V_0)$
is uniformly bounded in $L^2(Q_T)$.

\subsection{Approximate problem}

Recalling $T_k(s)=\max\{0,\min\{k,s\}\}$ for $s\in\R$,
we introduce the approximate problem
\begin{align}
  & \pa_t D_{0,k} = \diver(\na D_{0,k} + T_k(D_{0,k}) \na V_{0,k}), \label{4.Dk} \\
	& \lambda^2\Delta V_{0,k} = c_n e^{V_{0,k}}-c_p e^{-V_{0,k}} - T_k(D_{0,k}) + A(x)
	\quad\mbox{in }\Omega,\ t>0, \label{4.Vk} \\
	& V_{0,k} = \VD \mbox{ on }\Gamma_D,\quad 
	\na V_{0,k}\cdot\nu=0\mbox{ on }\Gamma_N, \label{4.bcV} \\
	& (\na D_{0,k} + T_k(D_{0,k})\na V_{0,k})\cdot\nu = 0\mbox{ on }\pa\Omega,\ t>0,
	\quad D_{0,k}(0) = D^I\mbox{ in }\Omega.	\label{4.bcD}
\end{align}
The existence of weak solutions to this problem can be proved similarly
as in Section \ref{sec.glob}. The only difference is the derivation of
an estimate for $V_k$ in the Lax--Milgram argument because of the nonlinear
Poisson equation. As the nonlinearity is monotone, the
$L^2$ norm of $\na V_k$ can be bounded in terms of the $L^2$ norm of $D_{0,k}$,
like in the proof of Lemma \ref{lem.exk}. 

To pass to the limit $k\to\infty$, we need additional estimates.
First, the free energy inequality gives the following bound, 
which can be also directly
proved from the Poisson equation using the test function $V_{0,k}-\VD$:
\begin{equation}\label{4.VkexpVk}
  \|V_{0,k}\exp(|V_{0,k}|)\|_{L^1(Q_T)}\le C.
\end{equation}
Second, we introduce the following function and its convex conjugate:
\begin{equation}\label{4.g}
  g(s) = \exp(s^2/4), \quad g^*(t) = \sup_{s>0}(st-g(s))\quad\mbox{for }s,t>0.
\end{equation}

\begin{lemma}\label{lem.Phi}
Let $\Phi(u):=u\sqrt{\log u}$ for $u\ge 1$.
Then there exists $C>0$ such that $\Phi(g^*(t))\le Ct\log t$ as $t\to\infty$.
\end{lemma}

\begin{proof}
It holds that $g^*(t)=s_t t-\exp(s_t^2/4)$, where $s_t>0$ is uniquely determined by
$s_t\exp(s_t^2/4)=2t$ for $t>0$. Thus,
for sufficiently large $t>0$, there exists $C>0$
such that $s_t\le C\sqrt{\log t}$, which implies that
$g^*(t)\le s_t t\le Ct\sqrt{\log t}$ for ``large'' values of $t$.
By definition of $\Phi$, we have
$$
  \Phi(g^*(t)) \le Ct\sqrt{\log t}\big[\log\big(Ct\sqrt{\log t}\big)\big]^{1/2}
	\quad\mbox{as }t\to\infty,
$$
and consequently, 
$$
  \frac{\Phi(g^*(t))}{t\log t}\le \bigg(\frac{\log(Ct\sqrt{\log t})}{\log t}
	\bigg)^{1/2} \le C\quad\mbox{as }t\to\infty,
$$
which proves the lemma.
\end{proof}

Lemma \ref{lem.Phi} and the $L^\infty(0,T;L^1(\Omega))$ bound for $D_{0,k}\log D_{0,k}$
from Lemma \ref{lem.eilimit} show that
\begin{equation}\label{4.gstar}
  \int_0^T\int_\Omega\Phi(g^*(D_{0,k}))dxdt 
	\le C+C\int_0^T\int_\Omega D_{0,k}\log D_{0,k} dxdt \le C.
\end{equation}
Since $\Phi(u)/u\to\infty$ as $u\to\infty$, we conclude from the 
De la Vall\'ee--Poussin lemma that $g^*(D_{0,k})$ is uniformly integrable
in $Q_T$. 

\subsection{Limit $k\to\infty$}\label{sec.prooflim}

The weak convergence of the potential, proved in Section \ref{sec.limitk}, does not
allow us to conclude the convergence $\exp(V_{0,k})\to\exp(V_0)$ as $k\to\infty$.
By exploiting the $L^1(Q_T)$ bound for $V_{0,k}\exp(|V_{0,k}|)$
and the monotonicity of the nonlinear terms in the Poisson equation, 
we are able to prove the strong convergence of $(V_{0,k})$.

\begin{lemma}
It holds that $V_{0,k}\to V_0$ strongly in $L^p(Q_T)$ for any $1\le p<\infty$.
\end{lemma}

\begin{proof}
We take the difference of the Poisson equation \eqref{4.Vk}, satisfied by
$V_{0,k}$ and $V_{0,\ell}$ for some $k$, $\ell\in\N$:
$$
  -\lambda^2\Delta(V_{0,k}-V_{0,\ell}) + c_n(e^{V_{0,k}}-e^{V_{0,\ell}})
	+ c_p(-e^{-V_{0,k}}+e^{-V_{0,\ell}}) = T_k(D_{0,k})-T_\ell(D_{0,\ell}).
$$
Then we choose the test function $(1+(V_{0,k}-V_{0,\ell})^2)^{-1/4}
(V_{0,k}-V_{0,\ell})$ in the weak formulation of the previous equation:
\begin{align*}
  \lambda^2\int_0^T&\int_\Omega\frac{2+(V_{0,k}-V_{0,\ell})^2}{2
	(1+(V_{0,k}-V_{0,\ell})^2)^{5/4}}|\na(V_{0,k}-V_{0,\ell})|^2 dxdt \\
	&\phantom{xx}{}+ c_n\int_0^T\int_\Omega\frac{(V_{0,k}-V_{0,\ell})
	(e^{V_{0,k}}-e^{V_{0,\ell}})}{(1+(V_{0,k}-V_{0,\ell})^2)^{1/4}}dxdt \\
	&\phantom{xx}{}+ c_p\int_0^T\int_\Omega\frac{(V_{0,k}-V_{0,\ell})
	(-e^{-V_{0,k}}+e^{-V_{0,\ell}})}{(1+(V_{0,k}-V_{0,\ell})^2)^{1/4}}dxdt \\
	&= \int_0^T\int_\Omega\frac{(V_{0,k}-V_{0,\ell})}{(1+(V_{0,k}-V_{0,\ell})^2)^{1/4}}
	(T_k(D_{0,k})-T_\ell(D_{0,\ell}))dxdt.
\end{align*}
Using
\begin{align*}
  c_n(&V_{0,k}-V_{0,\ell})(e^{V_{0,k}}-e^{V_{0,\ell}})
	+ c_p(V_{0,k}-V_{0,\ell})(-e^{-V_{0,k}}+e^{-V_{0,\ell}}) \\
	&\ge \min\{c_n,c_p\}(V_{0,k}-V_{0,\ell})\big((e^{V_{0,k}}-e^{V_{0,\ell}})
	+ (-e^{-V_{0,k}}+e^{-V_{0,\ell}})\big) \\
	&= 2\min\{c_n,c_p\}(V_{0,k}-V_{0,\ell})(\sinh(V_{0,k})-\sinh(V_{0,\ell})),
\end{align*}
we find that
\begin{align}\label{4.V0k}
 \lambda^2\int_0^T&\int_\Omega\frac{2+(V_{0,k}-V_{0,\ell})^2}{2
	(1+(V_{0,k}-V_{0,\ell})^2)^{5/4}}|\na(V_{0,k}-V_{0,\ell})|^2 dxdt \\
	&\phantom{xx}{}+ 2\min\{c_n,c_p\}\int_0^T\int_\Omega
	\frac{(V_{0,k}-V_{0,\ell})(\sinh(V_{0,k})-\sinh(V_{0,\ell}))}{
	(1+(V_{0,k}-V_{0,\ell})^2)^{1/4}}dxdt \nonumber \\
	&\le \int_0^T\int_\Omega\frac{|V_{0,k}-V_{0,\ell}|}{(1+(V_{0,k}-V_{0,\ell})^2)^{1/4}}
	|T_k(D_{0,k})-T_\ell(D_{0,\ell})|dxdt. \nonumber
\end{align}

We claim that the right-hand side converges to zero if $k$, $\ell\to\infty$.
For this, we decompose the right-hand side for some $L>1$ into two parts:
\begin{align*}
  J_1 &:= \int_0^T\int_{\{|V_{0,k}|+|V_{0,\ell}|\le L\}}
	\frac{|V_{0,k}-V_{0,\ell}|}{(1+(V_{0,k}-V_{0,\ell})^2)^{1/4}}
	|T_k(D_{0,k})-T_\ell(D_{0,\ell})|dxdt, \\
	J_2 &:= \int_0^T\int_{\{|V_{0,k}|+|V_{0,\ell}| > L\}}
	\frac{|V_{0,k}-V_{0,\ell}|}{(1+(V_{0,k}-V_{0,\ell})^2)^{1/4}}
	|T_k(D_{0,k})-T_\ell(D_{0,\ell})|dxdt.
\end{align*}
The integral $J_1$ converges to zero as $k$, $\ell\to\infty$, since
\begin{align}\label{4.J1}
  J_1 &\le \int_0^T\int_{\{|V_{0,k}|+|V_{0,\ell}|\le L\}}|V_{0,k}-V_{0,\ell}|
	|T_k(D_{0,k})-T_\ell(D_{0,\ell})|dxdt \\
	&\le L\int_0^T\int_\Omega|T_k(D_{0,k})-T_\ell(D_{0,\ell})|dxdt, \nonumber
\end{align}
and the strong convergence $D_{0,k}\to D_0$ strongly in $L^1(Q_T)$ implies that
$(D_{0,k})$ and also $(T_k(D_{0,k}))$ is a Cauchy sequence.
The difficult part is the limit $k$, $\ell\to\infty$ in $J_2$.

We infer from the Fenchel--Young inequality that
\begin{align*}
  J_2 &\le \int_0^T\int_{\{|V_{0,k}|+|V_{0,\ell}| > L\}}
	\big(1+|V_{0,k}-V_{0,\ell}|^{1/2}\big)\big(T_k(D_{0,k})+T_\ell(D_{0,\ell})\big)dxdt \\
	&\le 2\int_0^T\int_{\{|V_{0,k}|+|V_{0,\ell}| > L\}}
	g(1+|V_{0,k}-V_{0,\ell}|^{1/2})dxdt \nonumber \\
	&\phantom{xx}{}+ \int_0^T\int_{\{|V_{0,k}|+|V_{0,\ell}| > L\}}g^*(T_k(D_{0,k}))dxdt
	+ \int_0^T\int_{\{|V_{0,k}|+|V_{0,\ell}| > L\}}g^*(T_\ell(D_{0,\ell}))dxdt 
	\nonumber \\
	&=: J_{21} + J_{22} + J_{23}, \nonumber
\end{align*}
where $g$ and $g^*$ are defined in \eqref{4.g}. Elementary inequalities lead to
\begin{align*}
  J_{21} &\le 2\int_0^T\int_{\{|V_{0,k}|+|V_{0,\ell}| > L\}}
	\exp\bigg(\frac12(1+|V_{0,k}-V_{0,\ell}|)\bigg)dxdt \\
	&\le2 e^{1/2}\int_0^T\int_{\max\{|V_{0,k}|,|V_{0,\ell}|\} > L/2}
	e^{\max\{|V_{0,k}|,|V_{0,\ell}|\}}dxdt \\
	&\le \frac{4e^{1/2}}{L}\int_0^T\int_{\max\{|V_{0,k}|,|V_{0,\ell}|\} > L/2}
	\max\{|V_{0,k}|,|V_{0,\ell}|\}e^{\max\{|V_{0,k}|,|V_{0,\ell}|\}}dxdt
	\le \frac{C}{L},
\end{align*}
taking into account estimate \eqref{4.VkexpVk} in the last step.

Since $(V_{0,k})$ is bounded in $L^1(Q_T)$, there exists $C>0$ such that
for all $k$, $\ell\ge 1$, $\operatorname{meas}\{|V_{0,k}|+|V_{0,\ell}| > L\}\le C/L$.
We have already shown that \eqref{4.gstar} implies the uniform integrability of 
$g^*(D_k)$ in $Q_T$. 
Thus, for any $\eta>0$, there exists $L_\eta>1$ such that for all $L>L_\eta$,
$$
  \sup_{k,\ell\in\N}\int_0^T\int_{\{|V_{0,k}|+|V_{0,\ell}| > L\}}
	g^*(T_k(D_{0,k}))dxdt \le \eta,
$$
which means that $J_{22}+J_{23}\le 2\eta$ for all $k$, $\ell\in\N$.
This information as well as the estimate for $J_{21}$ yield
$J_2\le 3\eta$ for sufficiently large $L>1$. Then, together with estimate \eqref{4.J1}
for $J_1$, we obtain
$$
  \limsup_{k,\ell\to\infty}
	\int_0^T\int_\Omega\frac{|V_{0,k}-V_{0,\ell}|}{(1+(V_{0,k}-V_{0,\ell})^2)^{1/4}}
	|T_k(D_{0,k})-T_\ell(D_{0,\ell})|dxdt\le 4\eta
$$
for all $L>L_\eta$. Since $\eta>0$ is arbitrary, we conclude that
$$
  \lim_{k,\ell\to\infty}\int_0^T\int_\Omega\frac{|V_{0,k}-V_{0,\ell}|}{
	(1+(V_{0,k}-V_{0,\ell})^2)^{1/4}}|T_k(D_{0,k})-T_\ell(D_{0,\ell})|dxdt = 0.
$$

We perform the limit $k$, $\ell\to\infty$ in \eqref{4.V0k}, which shows that
$$
  \lim_{k,\ell\to\infty}\int_0^T\int_\Omega
	\frac{(V_{0,k}-V_{0,\ell})(\sinh(V_{0,k})-\sinh(V_{0,\ell}))}{
	(1+(V_{0,k}-V_{0,\ell})^2)^{1/4}}dxdt = 0.
$$
We deduce from the trigonometric addition formula
$$
  \sinh(a)-\sinh(b) = 2\sinh\frac{a-b}{2}\cosh\frac{a+b}{2}\quad\mbox{for }a,b\in\R
$$
that
$$
  \lim_{k,\ell\to\infty}\int_0^T\int_\Omega
	\frac{\sinh((V_{0,k}-V_{0,\ell})/2)}{V_{0,k}-V_{0,\ell}}\,
	\frac{\cosh((V_{0,k}+V_{0,\ell})/2)}{(1+(V_{0,k}-V_{0,\ell})^2)^{1/4}}
	(V_{0,k}-V_{0,\ell})^2 dxdt = 0.
$$
Taking into account that $\cosh(a)\ge 1$ for $a\in\R$ and that for every $0<p<\infty$,
there exists $c_p>0$ such that $\sinh(a)/a\ge c_p|a|^p$ for $a\in\R$, we conclude
that $(V_{0,k})$ is a Cauchy sequence in $L^p(Q_T)$ for every $p<\infty$
and consequently, $(V_{0,k})$ is convergent in that space.
\end{proof}

We proceed with the proof of Theorem \ref{thm.lim}. 
Estimate \eqref{4.VkexpVk} shows that
$$
  \int_0^T\int_\Omega\Psi(e^{\pm V_{0,k}})dxdt \le C,
$$
where $\Psi(u)=u\log u$ for $u\ge 0$. By the De la Vall\'ee--Poussin theorem,
$(\exp(\pm V_{0,k}))$ is uniformly integrable. The strong convergence $V_{0,k}\to V_0$
in $L^p(Q_T)$ implies, up to a subsequence, that 
$\exp(\pm V_{0,k})\to \exp(\pm V_0)$ a.e.\ in $Q_T$. Thus, 
$$
  \exp(\pm V_{0,k})\to \exp(\pm V_0) \quad\mbox{strongly in }L^1(Q_T).
$$
The proof in Section \ref{sec.limitk} shows that $D_{0,k}\to D_0$ strongly
in $L^1(Q_T)$, $\na V_{0,k}\rightharpoonup\na V_0$ and
$\na D_{0,k}+T_k(D_{0,k})\na V_{0,k}\rightharpoonup \na D_0+D_0\na V_0$ 
weakly in $L^1(Q_T)$. 
These convergence allows us to perform the limit $k\to\infty$ in 
\eqref{4.Dk}--\eqref{4.bcD}, showing that $(D_0,V_0)$ is a weak solution
to \eqref{1.limD}--\eqref{1.limbcV}.


\subsection{Weak-strong uniqueness for the limit problem}

We continue by proving the weak-strong uniqueness property. Let
$(D_0,V_0)$ be a bounded strong solution and $(D,V)$ be a weak solution
to \eqref{1.limD}--\eqref{1.limbcV} satisfying the assumptions of Theorem \ref{thm.lim}.
The proof is based on the relative free energy
\begin{align*}
  & H[(D,V)|(D_0,V_0)] = H_1[D|D_0] + H_2[V|V_0], \quad\mbox{where} \\
	& H_1[D|D_0] = \int_\Omega\bigg(D\log\frac{D}{D_0} - D + D_0\bigg)dx, \\
	& H_2[V|V_0] = \int_\Omega\bigg(\frac{\lambda^2}{2}|\na(V-V_0)|^2
	+ c_n e^{V_0} f_0(V-V_0) + c_p e^{-V_0} f_0(V_0-V)\bigg)dx,
\end{align*}
recalling that $f_0(s)=(s-1)e^s+1$ for $s\in\R$. The proof is divided into several 
steps. 

{\em Step 1: Estimates for the potential $V-V_0$.} 
We wish to derive a bound for $V-V_0$ in terms of the relative 
free energy $H_1[D|D_0]$.
To this end, we use the test function $V-V_0$ in the weak formulation of the
difference of the equations satisfied by $V$ and $V_0$, respectively:
\begin{align*}
  \lambda^2&\int_\Omega|\na(V-V_0)|^2 dx 
	+ \int_\Omega\big(c_n(e^{V}-e^{V_0}) - c_p(e^{-V}-e^{-V_0})\big)(V-V_0)dx \\
	&= \int_\Omega(V-V_0)(D-D_0)dx.
\end{align*}
Since $V_0$ is bounded by assumption, we can estimate the second integral 
on the left-hand side according to 
\begin{align*}
  \big(c_n&(e^{V}-e^{V_0}) - c_p(e^{-V}-e^{-V_0})\big)(V-V_0) \\
	&= c_n e^{V_0}(e^{V-V_0}-1)(V-V_0) + c_p e^{-V_0}(e^{V_0-V}-1)(V_0-V) \\
	&\ge c|V-V_0|(e^{|V-V_0|}-1),
\end{align*}
where $c>0$ depends on the $L^\infty(Q_T)$ norm of $V_0$. We infer that
\begin{equation}\label{4.aux2}
  \int_\Omega|\na(V-V_0)|^2 dx + \int_\Omega|V-V_0|(e^{|V-V_0|}-1)dx
	\le C\int_\Omega|V-V_0||D-D_0|dx.
\end{equation}
Let $g_\xi(s) = \xi s(e^s-1)$ for $\xi>0$, $s\ge 0$ and $g_\xi^*$ be its convex
conjugate. We deduce from the Fenchel--Young inequality and Lemma \ref{lem.conj} 
in the Appendix that
\begin{align*}
  \int_\Omega&|V-V_0||D-D_0|dx \le \int_\Omega g_\xi(|V_0-V|)dx
	+ \int_\Omega g_\xi^*(|D-D_0|)dx \\
	&\le \xi\int_\Omega|V-V_0|(e^{|V-V_0|}-1)dx
	+ \xi\int_\Omega\frac{(\log(1+|D-D_0|/\xi))^2}{1+\log(1+|D-D_0|/\xi)}
	\bigg(1 + \frac{|D-D_0|}{\xi}\bigg)dx.
\end{align*}
For $0<\xi<1$, the first term on the right-hand side can be absorbed by the
left-hand side of \eqref{4.aux2}, leading to
\begin{align*}
  \int_\Omega&|\na(V-V_0)|^2 dx + \int_\Omega|V-V_0|(e^{|V-V_0|}-1)dx \\
	&\le C \xi\int_\Omega\frac{(\log(1+|D-D_0|/\xi))^2}{1+\log(1+|D-D_0|/\xi)}
	\bigg(1 + \frac{|D-D_0|}{\xi}\bigg)dx.
\end{align*}

We claim that the right-hand side can be controlled by $H_1[D|D_0]$. In fact,
we claim that for $0<\gamma_0\le D_0\le \gamma_1$ and $D\ge 0$,
\begin{equation}\label{4.ineq}
  \frac{(\log(1+|D-D_0|/\xi))^2}{1+\log(1+|D-D_0|/\xi)}
	\bigg(1 + \frac{|D-D_0|}{\xi}\bigg) \le C(\xi,\gamma_0,\gamma_1)
	\bigg(D\log\frac{D}{D_0}-D+D_0\bigg).
\end{equation}
This can be seen by analyzing the behavior of both sides of \eqref{4.ineq}
for $D\to 0$, $D\to D_0$, and $D\to\infty$. For $D\to 0$, the left-hand side
of \eqref{4.ineq} remains bounded, while the right-hand side is uniformly positive.
For $D\to\infty$, both sides diverge like $D\log D$. Finally, for $D\to D_0$,
a Taylor expansion shows that both sides tend to zero quadratically in $|D-D_0|$. 
We conclude that
\begin{equation}\label{4.estH1}
  \int_\Omega|\na(V-V_0)|^2dx + \int_\Omega|V-V_0|(e^{|V-V_0|}-1)dx
	\le C H_1[D|D_0].
\end{equation}

{\em Step 2: Estimate for $(dH_1/dt)[D|D_0]$.} We differentiate $H_1[D|D_0]$ with
respect to time:
\begin{align*}
  \frac{dH_1}{dt}[D|D_0] &= \bigg\langle\pa_t D,\log\frac{D}{D_0}\bigg\rangle
	+ \bigg\langle\pa_t D_0,1-\frac{D}{D_0}\bigg\rangle \\
	&= -\int_\Omega\na\log\frac{D}{D_0}\cdot(\na D+D\na V)dx
	- \int_\Omega\na\bigg(1-\frac{D}{D_0}\bigg)\cdot(\na D_0+D_0\na V_0)dx,
\end{align*}
Elementary computations lead to
\begin{equation}\label{4.dH1dt}
  \frac{dH_1}{dt}[D|D_0] = -\int_\Omega D\bigg|\na\log\frac{D}{D_0}\bigg|^2 dx
	- \int_\Omega D\na\log\frac{D}{D_0}\cdot\na(V-V_0)dx.
\end{equation}

To reformulate the last integral, 
we use $V-V_0$ as a test function in the weak formulation
of the difference of the equations satisfied by $D$ and $D_0$, respectively:
\begin{align*} 
  \langle\pa_t&(D-D_0),V-V_0\rangle + \int_\Omega\na(V-V_0)\cdot\na(D-D_0)dx
	+ \int_\Omega D|\na(V-V_0)|^2 dx \\
	&= -\int_\Omega(D-D_0)\na V_0\cdot\na(V-V_0)dx. 
\end{align*}
Rewriting the second term on the left-hand side,
\begin{align*}
  \int_\Omega&\na(V-V_0)\cdot\na(D-D_0)dx
	= \int_\Omega\na(V-V_0)\cdot\bigg(\bigg(\frac{D}{D_0}-1\bigg)\na D_0
	+ D\na\log\frac{D}{D_0}\bigg)dx \\
	&= \int_\Omega(D-D_0)\frac{\na D_0}{D_0}\cdot\na(V-V_0)dx
	+ \int_\Omega D\na(V-V_0)\cdot\na\log\frac{D}{D_0}dx,
\end{align*}
we find that
\begin{align*}
  \langle\pa_t&(D-D_0),V-V_0\rangle + \int_\Omega D|\na(V-V_0)|^2 dx \\
	&= -\int_\Omega(D-D_0)\na(V-V_0)\cdot\bigg(\frac{\na D_0}{D_0}+\na V_0\bigg)dx
	- \int_\Omega D\na(V-V_0)\cdot\na\log\frac{D}{D_0}dx.
\end{align*}
We add this expression to \eqref{4.dH1dt}:
\begin{align*}
  \frac{dH_1}{dt}&[D|D_0] + \langle\pa_t(D-D_0),V-V_0\rangle 
	+ \int_\Omega D\bigg|\na\log\frac{D}{D_0}\bigg|^2 dx 
	+ \int_\Omega D|\na(V-V_0)|^2 dx \\
	&{}+ 2\int_\Omega D\na(V-V_0)\cdot\na\log\frac{D}{D_0}dx
	= -\int_\Omega(D-D_0)\na(V-V_0)\cdot\bigg(\frac{\na D_0}{D_0}+\na V_0\bigg)dx.
\end{align*}
The last three terms of the left-hand side can be written as a square, leading to
\begin{align}\label{4.aux3}
  \frac{dH_1}{dt}&[D|D_0] + \langle\pa_t(D-D_0),V-V_0\rangle 
	+ \int_\Omega D\bigg|\na\bigg(\log\frac{D}{D_0}+V-V_0\bigg)\bigg|^2 dx \\
	&= -\int_\Omega(D-D_0)\na(V-V_0)\cdot\bigg(\frac{\na D_0}{D_0}+\na V_0\bigg)dx
	\nonumber
\end{align}

We estimate the right-hand side by decomposing the integral in two terms,
$$
  I_\pm := -\int_\Omega(D-D_0)_\pm\na(V-V_0)\cdot
	\bigg(\frac{\na D_0}{D_0}+\na V_0\bigg)dx,
$$
recalling that $z_+=\max\{0,z\}$ and $z_-=\min\{0,z\}$. 
In the integral $I_+$, we consider ``large'' values of $D$, i.e.\ 
$D>c:=\inf_{Q_T}D_0>0$,
while ``small'' values of $D$, i.e.\ $0\le D\le c$, are taken into account in $I_-$.

First, using Young's inequality and the assumptions
$\na\log D_0$, $\na V_0\in L^\infty(Q_T)$:
\begin{align}\label{4.I+est}
  I_+ &= -\int_\Omega(D-D_0)_+\na\bigg(\log\frac{D}{D_0}+V-V_0\bigg)\cdot
	\bigg(\frac{\na D_0}{D_0}+\na V_0\bigg)dx \\
	&\phantom{xx}{}+ \int_\Omega(D-D_0)_+\na\log\frac{D}{D_0}\cdot
	\bigg(\frac{\na D_0}{D_0}+\na V_0\bigg)dx \nonumber \\
	&\le \frac12\int_\Omega D\bigg|\na\bigg(\log\frac{D}{D_0}+V-V_0\bigg)\bigg|^2 dx 
	+ C(D_0,V_0)\int_\Omega\frac{1}{D}(D-D_0)_+^2 dx \nonumber \\
	&\phantom{xx}{}+ \int_\Omega(D-D_0)_+\na\log\frac{D}{D_0}\cdot
	\bigg(\frac{\na D_0}{D_0}+\na V_0\bigg)dx. \nonumber
\end{align}
Observe that the positive part avoids the singularity since
$(D-D_0)_+^2/D=0$ if $D\le c$. Taylor's formula yields
$$
  D\log\frac{D}{D_0}-D+D_0 
	= D_0\bigg(\frac{D}{D_0}\bigg(\log\frac{D}{D_0}-1\bigg)+1\bigg)
	= \frac{D_0}{2\xi}\bigg(\frac{D}{D_0}-1\bigg)^2
  \ge \frac{D_0^2}{2D}\bigg(\frac{D}{D_0}-1\bigg)^2
$$
for some $1\le\xi\le D/D_0$. Then the second integral on the right-hand side of 
\eqref{4.I+est} becomes
$$
  \int_\Omega\frac{1}{D}(D-D_0)_+^2 dx 
	= \int_{\{D>D_0\}}\frac{D_0^2}{D}\bigg(\frac{D}{D_0}-1\bigg)^2 dx
	\le 2H_1[D|D_0].
$$
The last integral in \eqref{4.I+est} is formulated as
\begin{align*}
  I_1 &:= \int_\Omega\bigg(1-\frac{D_0}{D}\bigg)_+\na\frac{D}{D_0}
	\cdot(\na D_0+D_0\na V_0)dx \\
	&= \int_\Omega\na F\bigg(\frac{D}{D_0}\bigg)\cdot(\na D_0+D_0\na V_0)dx,
\end{align*}
where $F(s) = (s-1-\log s)\mathrm{1}_{\{s>1\}}\ge 0$ for $s>0$. 
The no-flux boundary conditions allow us to integrate by parts:
\begin{equation*}
  I_1	= -\big\langle\diver(\na D_0+D_0\na V_0),F(D/D_0)\big\rangle
	= -\langle \pa_t D_0,F(D/D_0)\rangle.
\end{equation*}
By our assumption $\pa_t D_0\in L^1(0,T;L^\infty(\Omega))$ and the property
$F(s)\le s(\log s-1)+1$ for $s\ge 0$, we conclude that there exists
$\gamma_1\in L^1(0,T)$ such that
$$
  I_1 \le \gamma_1(t)\int_\Omega F(D/D_0)dx 
	\le \gamma_1(t)H_1[D|D_0].
$$
Therefore, setting $\gamma_2(t)=\gamma_1(t)+2C(D_0,V_0)$,
we deduce from \eqref{4.I+est} that
\begin{equation}\label{4.I+}
  I_+ \le \frac12\int_\Omega D\bigg|\na\bigg(\log\frac{D}{D_0}+V-V_0\bigg)\bigg|^2 dx 
	+ \gamma_2(t) H_1[D,D_0].
\end{equation}

Now, we estimate $I_-$. By our assumptions on $D_0$ and $V_0$, we compute
\begin{align*}
  I_- &\le C(D_0,V_0)\int_\Omega|(D-D_0)_-||\na(V-V_0)|dx \\
	&\le C(D_0,V_0)\int_\Omega(D-D_0)_-^2 dx + C(D_0,V_0)\int_\Omega|\na(V-V_0)|^2 dx.
\end{align*}
The first integral can be bounded from above by $CH_1[D|D_0]$ since
$D\log(D/D_0)-D+D_0\ge(D-D_0)^2/(2D_0)$ for $D<D_0$. 
The second integral is estimated from above by $CH_2[V,V_0]$. We conclude that
\begin{equation}\label{4.I-}
  I_- \le C(D_0,V_0)\big(H_1[D|D_0] + H_2[V|V_0]\big).
\end{equation}

Inserting estimates \eqref{4.I+} for $I_+$ and \eqref{4.I-} for $I_-$ into
\eqref{4.aux3}, we obtain for some $\gamma_3\in L^1(0,T)$,
\begin{align}\label{4.aux5}
  \frac{dH_1}{dt}&[D|D_0] + \langle\pa_t(D-D_0),V-V_0\rangle 
	+ \frac12\int_\Omega D\bigg|\na\bigg(\log\frac{D}{D_0}+V-V_0\bigg)\bigg|^2 dx \\
	&\le \gamma_3(t)H[(D,V)|(D_0,V_0)]. \nonumber
\end{align}

{\em Step 3: Estimate of $\langle\pa_t(D-D_0),V-V_0\rangle$.}
The difference $V-V_0$ satisfies the Poisson equation
$$
  D-D_0 = -\lambda^2\Delta(V-V_0) + c_n(e^V-e^{V_0}) - c_p(e^{-V}-e^{-V_0}).
$$
Thus, replacing $D-D_0$ in $\langle\pa_t(D-D_0),V-V_0\rangle$
by the right-hand side and integrating by parts in the term
involving $\Delta(V-V_0)$ leads to
\begin{align}\label{4.aux4}
  \langle\pa_t(D-D_0),V-V_0\rangle
	&= \frac{\lambda^2}{2}\frac{d}{dt}\int_\Omega|\na(V-V_0)|^2 dx \\
	&\phantom{xx}{}+ c_n\langle\pa_t(e^V-e^{V_0}),V-V_0\rangle 
	- c_p\langle\pa_t(e^{-V}-e^{-V_0}),V-V_0\rangle. \nonumber
\end{align} 
The second term on the right-hand side can be reformulated according to
\begin{align*}
  \langle\pa_t&(e^V-e^{V_0}),V-V_0\rangle 
	= \langle\pa_t f_0(V-V_0),e^{V_0}\rangle 
	+ \langle\pa_t(e^{V_0}),(V-V_0)(e^{V-V_0}-1)\rangle \\
	&= \frac{d}{dt}\int_\Omega e^{V_0} f_0(V-V_0)dx
	+ \big\langle\pa_t(e^{V_0}),(V-V_0)(e^{V-V_0}-1)-f_0(V-V_0)\big\rangle \\
	&= \frac{d}{dt}\int_\Omega e^{V_0} f_0(V-V_0)dx
	+ \big\langle\pa_t(e^{V_0}),e^{V-V_0}-(V-V_0)-1\big\rangle.
\end{align*}
Our assumption on $V_0$ implies that $\pa_t(e^{V_0})\in L^1(0,T;L^\infty(\Omega))$
such that
\begin{align*}
  \langle\pa_t&(e^V-e^{V_0}),V-V_0\rangle 
	\ge \frac{d}{dt}\int_\Omega e^{V_0} f_0(V-V_0)dx
	- \gamma_4(t)\int_\Omega(e^{V-V_0}-(V-V_0)-1)dx
\end{align*}
for some nonnegative function $\gamma_4(t)$. It holds that
$$
  e^s-s-1\le C|s|(e^{|s|}-1)\quad\mbox{for }s\in\R,
$$
since both sides behave like $s^2$ as $s\to 0$ and for $|s|\to\infty$,
the right-hand side tends to infinity faster than the left-hand side. 
Therefore, we deduce from \eqref{4.estH1} that
$$
  \langle\pa_t(e^V-e^{V_0}),V-V_0\rangle 
	\ge \frac{d}{dt}\int_\Omega e^{V_0} f_0(V-V_0)dx
	- (\gamma_4(t)+C)H[(D,V)|(D_0,V_0)].
$$
In a similar way, it follows that
$$
  -\langle\pa_t(e^{-V}-e^{-V_0}),V-V_0\rangle 
	\ge \frac{d}{dt}\int_\Omega e^{-V_0} f_0(V_0-V)dx
	- (\gamma_4(t)+C)H[(D,V)|(D_0,V_0)].
$$
Inserting these inequalities into \eqref{4.aux4} and taking into account the
definition $H_2$ shows that, for some $\gamma_5\in L^1(0,T)$,
$$
  \langle\pa_t(D-D_0),V-V_0\rangle \ge \frac{dH_2}{dt}[V|V_0]
	- \gamma_5(t)H[(D,V)|(D_0,V_0)].
$$

{\em Step 4: Conclusion.} We infer from the previous inequality and \eqref{4.aux5}
that
$$
  \frac{dH}{dt}[(D,V)|(D_0,V_0)] \le (\gamma_3(t)+\gamma_5(t))H[(D,V)|(D_0,V_0)]
$$
for $0<t<T$. As $\gamma_3+\gamma_5\in L^1(0,T)$ and $H[(D,V)|(D_0,V_0)]=0$ at $t=0$,
Gronwall's lemma implies that $H[(D,V)|(D_0,V_0)](t)=0$ for $0<t<T$, which
gives $D=D_0$ and $V=V_0$ in $Q_T$ and finishes the proof.


\section{Proof of Theorem \ref{thm.eps}}\label{sec.eps}

Since there is no factor $\eps$ in the equation for $D_\eps$, we can proceed
as in the existence proof and show, using the Aubin--Lions lemma, that
up to a subsequence,
\begin{equation}\label{5.Deps}
  D_\eps\to D_0 \quad\mbox{strongly in }L^1(Q_T)\mbox{ as }\eps\to 0.
\end{equation}
The assumption on the boundary data $\nD $ and $\pD $ implies that
$\Lambda_\eps=0$ (see \eqref{1.lambda}). Therefore, by the free energy inequality
\eqref{1.ei},
$$
  \int_0^T\int_\Omega \big(n_\eps|\na(\log n_\eps-V_\eps)|^2 
	+ p_\eps|\na(\log p_\eps+V_\eps)|^2\big)dxdt \le H^I\eps.
$$
It follows that
$$
  2\na\sqrt{n_\eps}-\sqrt{n_\eps}\na V_\eps
  = \sqrt{n_\eps}\na(\log n_\eps-V_\eps)\to 0 \quad
	\mbox{strongly in }L^2(Q_T).
$$
Furthermore, since $\sqrt{n_\eps}$
is uniformly bounded in $L^\infty(0,T;L^2(\Omega))$, we find that
$$
  \na n_\eps - n_\eps\na V_\eps 
	= \sqrt{n_\eps}\big(2\na\sqrt{n_\eps}-\sqrt{n_\eps}\na V_\eps\big)\to 0
	\quad\mbox{strongly in }L^1(Q_T). 
$$

By estimate \eqref{v.infty}, which holds in two space dimensions,
$$
  \|V_\eps\|_{L^\infty(\Omega)}
	\le C\big(1+\|(n_\eps-p_\eps-D_\eps+A(x))\log|n_\eps-p_\eps-D_\eps+A(x)|
	\|_{L^1(\Omega)}\big).
$$
In view of inequality \eqref{1.ei}, 
this gives a uniform $L^\infty(Q_T)$ bound for $V_\eps$. We infer that
$$
  \na(n_\eps e^{-V_\eps}) 
	= e^{-V_\eps}(\na n_\eps - n_\eps\na V_\eps)\to 0\quad\mbox{strongly in }L^1(Q_T).
$$
By Poincar\'e's inequality, since $\nD e^{-\VD}=c_n=\mbox{const.}$,
\begin{equation}\label{5.neps}
  \|n_\eps e^{-V_\eps}-\nD e^{-\VD}\|_{L^1(Q_T)}
	\le C\|\na(n_\eps e^{-V_\eps})\|_{L^1(Q_T)} + C\|\na(\nD e^{-\VD})\|_{L^1(Q_T)}
	\to 0.
\end{equation}
Similarly, it follows that
\begin{equation}\label{5.peps}
  p_\eps e^{V_\eps}-\pD e^{\VD}\to 0 \quad\mbox{strongly in }L^1(Q_T).
\end{equation}

Next, we reformulate the Poisson equation as
$$
  \lambda^2\Delta V_\eps = e^{V_\eps}(n_\eps e^{-V_\eps})
	- e^{-V_\eps}(p_\eps e^{V_\eps}) - D_\eps + A
	= e^{V_\eps}(\nD e^{-\VD}) - e^{-V_\eps}(\pD e^{\VD}) - D_\eps + A + E_\eps,
$$
where
$$
  E_\eps := e^{V_\eps}(n_\eps e^{-V_\eps}-\nD e^{-\VD})
	- e^{-V_\eps}(p_\eps e^{V_\eps} - \pD e^{\VD})
$$
is an error term. Then $V_\eps-V_{\eps'}$ for some $\eps'>0$ solves
$$
  \lambda^2\Delta(V_\eps-V_{\eps'}) = \nD e^{-\VD}(e^{V_\eps}-e^{V_{\eps'}})
	- \pD e^{\VD}(e^{-V_\eps}-e^{-V_{\eps'}}) - (D_\eps-D_{\eps'}) + E_\eps-E_{\eps'},
$$
and choosing the test function $V_\eps-V_{\eps'}$ in the weak formulation, we have
\begin{align*}
  \lambda^2&\int_\Omega|\na(V_\eps-V_{\eps'})|^2 dx
	= -\int_\Omega\nD e^{-\VD}(e^{V_\eps}-e^{V_{\eps'}})(V_\eps-V_{\eps'})dx \\
	&\phantom{xx}{}
	- \int_\Omega \pD e^{\VD}(-e^{-V_\eps}+e^{-V_{\eps'}})(V_\eps-V_{\eps'})dx
	+ \int_\Omega(D_\eps-D_{\eps'})(V_\eps-V_{\eps'})dx \\
	&\phantom{xx}{}- \int_\Omega(E_\eps-E_{\eps'})(V_\eps-V_{\eps'})dx \\
	&\le \int_\Omega(D_\eps-D_{\eps'})(V_\eps-V_{\eps'})dx
	- \int_\Omega(E_\eps-E_{\eps'})(V_\eps-V_{\eps'})dx,
\end{align*}
because of the monotonicity of $z\mapsto e^z$ and $z\mapsto -e^{-z}$. 
The strong convergences \eqref{5.neps} and \eqref{5.peps} 
as well as the uniform $L^\infty(Q_T)$ bound of $V_\eps$ imply that
$E_\eps\to 0$ strongly in $L^1(Q_T)$ as $\eps\to 0$. Therefore, in view
of \eqref{5.Deps},
$$
  \lambda^2\|\na(V_\eps-V_{\eps'})\|_{L^2(Q_T)} 
	\le \|V_\eps-V_{\eps'}\|_{L^\infty(Q_T)}\big(\|D_\eps-D_{\eps'}\|_{L^1(Q_T)}
	+ \|E_\eps-E_{\eps'}\|_{L^1(Q_T)}\big)\to 0
$$
as $\eps$, $\eps'\to 0$. Taking into account the Poincar\'e inequality, we
obtain $V_\eps-V_{\eps'}\to 0$ strongly in $L^2(0,T;H^1(\Omega))$ as
$\eps,\eps'\to 0$. This means that $(V_\eps)$ is a Cauchy sequence in
$L^2(0,T;H^1(\Omega))$ and consequently, there exists a function 
$V_0\in L^2(0,T;H^1(\Omega))$ such that
\begin{equation}\label{5.Veps}
  V_\eps\to V_0 \quad\mbox{strongly in }L^2(0,T;H^1(\Omega)).
\end{equation}
Thus, up to a subsequence, $V_\eps\to V_0$ and $\exp(V_\eps)\to\exp(V_0)$
a.e.\ in $Q_T$. Because of the uniform bound for $(V_\eps)$ in $L^\infty(Q_T)$,
this shows that $\exp(V_\eps)\rightharpoonup\exp(V_0)$ weakly* in $L^\infty(Q_T)$.
Then we infer from \eqref{5.neps} that
\begin{align*}
  n_\eps = e^{V_\eps}(n_\eps e^{-V_\eps}) \to e^{V_0}(\nD e^{-\VD})
	= \nD e^{V_0-\VD} =: n_0 &\quad\mbox{a.e. in }Q_T, \\
	n_\eps = e^{V_\eps}(n_\eps e^{-V_\eps}) \rightharpoonup e^{V_0}(\nD e^{-\VD})
	= n_0 &\quad\mbox{weakly in }L^1(Q_T).
\end{align*}
By the Dunford--Pettis theorem, $(n_\eps)$ is uniformly integrable. Thus, the
a.e.\ convergence of $(n_\eps)$ implies that
$$
  n_\eps\to n_0\quad\mbox{strongly in }L^1(Q_T)
$$
and by similar arguments,
$$
  p_\eps\to p_0 :=\pD e^{\VD-V_0}\quad\mbox{strongly in }L^1(Q_T).
$$

Finally, we perform the limit $\eps\to 0$ in \eqref{1.n}--\eqref{1.V} and
prove that $(n_0,p_0,D_0,V_0)$ satisfies the limit problem. The only delicate
limit is in the flux term in the equation for $D_\eps$. The proof is similar
as in Section \ref{sec.limitk}. Indeed, the bound on the entropy production for
$D_\eps$ in \eqref{1.ei} shows that, up to a subsequence,
$$
  2\na\sqrt{D_\eps}+\sqrt{D_\eps}\na V_\eps \rightharpoonup \chi
	\quad\mbox{weakly in }L^2(Q_T)
$$
for some $\chi\in L^2(Q_T)$. 
We deduce from \eqref{5.Deps} and \eqref{5.Veps} that
\begin{align*}
  & \sqrt{D_\eps}\na V_\eps\to \sqrt{D_0}\na V_0 \quad\mbox{strongly in }L^1(Q_T), \\
  & \big\|\na\big(\sqrt{D_\eps}-\sqrt{D_0}\big)\big\|_{L^2(0,T;H^{-1}(\Omega))}
	\le \big\|\sqrt{D_\eps}-\sqrt{D_0}\big\|_{L^2(Q_T)} \to 0.
\end{align*}
Thus, we can identify $\chi=2\na\sqrt{D_0}+\sqrt{D_0}\na V_0$. This relation,
together with \eqref{5.Deps}, yields
$$
  \na D_\eps + D_\eps\na V_\eps 
	= \sqrt{D_\eps}\big(2\na\sqrt{D_\eps}+\sqrt{D_\eps}\na V_\eps\big)
	\rightharpoonup \sqrt{D_0}\chi = \na D_0+D_0\na V_0
$$
weakly in $L^1(Q_T)$. 
This implies, as at the end of Section \ref{sec.limitk}, that 
$\pa_t D_\eps\rightharpoonup \pa_t D_0$ weakly in $L^1(0,T;H^{s}(\Omega)')$
for $s>1+d/2$. The proof is finished.


\section{Numerical illustrations}\label{sec.num}

We present numerical simulations of the full model \eqref{1.n}--\eqref{1.nbcd}
and the reduced model \eqref{1.limD}--\eqref{1.limbcV} in one space dimension
to illustrate the behavior of the solutions and to compare the results
with those from \cite{SBW09}.

\subsection{Numerical scheme}

We assume that $\Omega=(0,L)$ for
some $L>0$ and we impose Dirichlet boundary conditions for $n$, $p$, and $V$.
For the numerical discretization, we formulate the reduced model 
in terms of the quasi-Fermi potentials
$$
  \phi_{n} = -\log n_0+V_0, \quad 
  \phi_{p} = \log p_0+V_0, \quad \phi_{D} = \log D_0+V_0.
$$
The reduced system reads as
\begin{align}
  & \pa_x J_{n,0} = \pa_x J_{p,0} = 0, \quad 
  \pa_t D_{0} + \pa_x J_{D,0} = 0, \quad
  \lambda^2\pa_{xx}V_0 = n_0-p_0-D_0+A(x), \label{6.eq1} \\
	& J_{n,0} = -e^{V_0-\phi_{n}}\pa_x\phi_{n}, \quad 
  J_{p,0} = -e^{\phi_{p}-V_0}\pa_x\phi_{p}, \quad
	J_{D,0} = -e^{\phi_{D}-V_0}\pa_x\phi_{D}, \label{6.eq2}
\end{align}
in $(0,L)$, $t>0$, with the initial and boundary conditions
\begin{align*}
  & \phi_{n}(0,t)=\phi_{p}(0,t)=U_0, \quad 
  \phi_{n}(L,t)=\phi_{p}(L,t)=U_L, \\
	& \pa_x\phi_{D}(0,t)=\pa_x\phi_{D}(L,t)=0, \\
	& V_0(0,t)=V_{\rm bi}+U_0, \quad
	V_0(L,t)=V_{\rm bi}+U_L \quad\mbox{for }t>0, \\
	& D_0(x,0)=D^I(x) \quad\mbox{for }x\in(0,L).
\end{align*}
Here, $U_0$ and $U_L$ are two (possibly time-dependent) applied potentials 
at the electrodes, and $V_{\rm bi}$ is the built-in potential, which is the
potential that corresponds to the thermal-equilibrium densities \cite{Jue09}:
$$
  V_{\rm bi} = \log\bigg(\frac12\big(D_e-A+\sqrt{(D_e-A)^2+4}\big)\bigg),
$$
where $D_e$ is the dopant concentration at the electrodes. Moreover, the
initial data for the electrons and holes are given by $n^I=\exp(V^I)$ and
$p^I=\exp(-V^I)$, where $V^I$ is the solution to the Poisson equation with
the above boundary conditions.

The scaled Debye length is given by $\lambda^2=\eps_s U_T/(qL^2n_i)$,
where the meaning and the values of the physical parameters are as follows:
\begin{itemize}
\item semiconductor permittivity of silicon: $\eps_s=8.85\cdot 10^{-13}$ As/Vcm;
\item thermal voltage at 300\,K: $U_T=0.026$ V;
\item elementary charge: $q=1.6\cdot 10^{-19}$ As;
\item device length: $L=50$ nm;
\item (reference) intrinsic density: $n_i=2\cdot 10^{19}$ cm$^{-3}$;
\item reference current density $J_0=400$ Acm$^{-2}$.
\end{itemize}
These values are similar to those in \cite{SBW09}, and they lead to
$\lambda^2=2.86\cdot 10^{-4}$. Furthermore, we choose as in \cite{SBW09}
constant scaled doping concentrations, $D^I=2.5$, $A=0.25$, and $D_e = 25$. 

Equations \eqref{6.eq1}--\eqref{6.eq2} are discretized by the finite-volume 
method. More precisely, the continuity equations are approximated by a
Scharfetter--Gummel scheme introduced in \cite{ScGu69}. For instance,
discretizing $(0,L)$ by $x_1=0<x_2<\cdots<x_N=L$, the continuity equation
for the electrons becomes $J_{n,k+1/2}-J_{n,k-1/2}=0$ on the control volume
$\omega_{k}=((x_{k-1}+x_k)/2,(x_k+x_{k+1})/2)$, where
$$
  J_{n,k+1/2} = \frac{1}{x_{k+1}-x_k}\big(B(V_k-V_{k+1})e^{V_k-\phi_{n,k}}
	- B(V_{k+1}-V_k)e^{V_{k+1}-\phi_{n,k+1}}\big),
$$
$B(s)=s/(e^s-1)$ is the Bernoulli function, $J_{k+1/2}$ approximates $J_n$ in
$\omega_{k}$, and $(\phi_{n,k},V_k)$ approximates $(\phi_n,V)(x_k)$. 
The continuity equation
for $D$ is discretized by the implicit Euler method. At each time step,
we use Newton's method to solve the discrete nonlinear system of $4N$ variables,
using the solution from the previous time step as the initial guess.

\subsection{Limit $\eps\to 0$}

The first numerical test is concerned with the behavior of the solutions
to the full system \eqref{1.n}--\eqref{1.nbcd} when $\eps\to 0$.
We consider only the equilibrium case when the applied voltage vanishes,
$U_0=U_L=0$.
Figure \ref{fig.eps1} illustrates the charge densities at times
$t=T_f/10$ and $t=T_f$, where $T_f=0.1$ corresponds to approximately
100\,ps. The time $T_f$ is
chosen in such a way that the solution at $t=T_f$ is close to the steady state. 
Note that we present the densities in the interval $[0.1,0.9]$
to avoid the boundary layers (e.g., Figure \ref{fig.eps2} left shows the boundary
layers for the oxygene vacancy density).

We see that the densities converge for $\eps\to 0$ to the densities associated
with the reduced system \eqref{1.limD}--\eqref{1.limbcV}, confirming
the results from Theorem \ref{thm.lim}. The values
for the densities do not vary much in space since we have chosen constant
doping concentrations. Figure \ref{fig.eps2} (right) shows that the convergence
is linear, i.e., $\|D_\eps-D_0\|_{L^1}\le C\eps$.
This is expected since the parameter $\eps$ appears in \eqref{1.n} and \eqref{1.p}
with first order. A rigorous proof, however, is delicate as the regularity
of solutions to the full model is rather low.

\begin{figure}[ht]
\begin{center}
\includegraphics[scale=0.45]{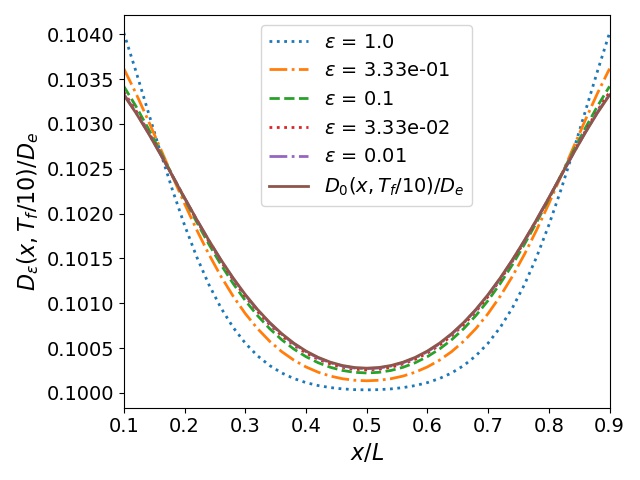}
\includegraphics[scale=0.45]{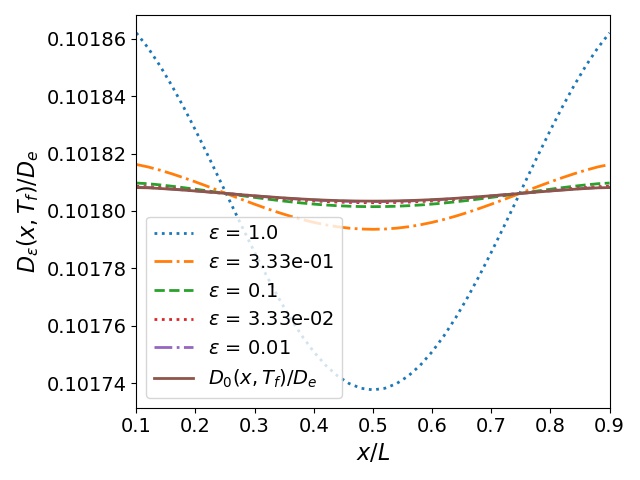}
\includegraphics[scale=0.45]{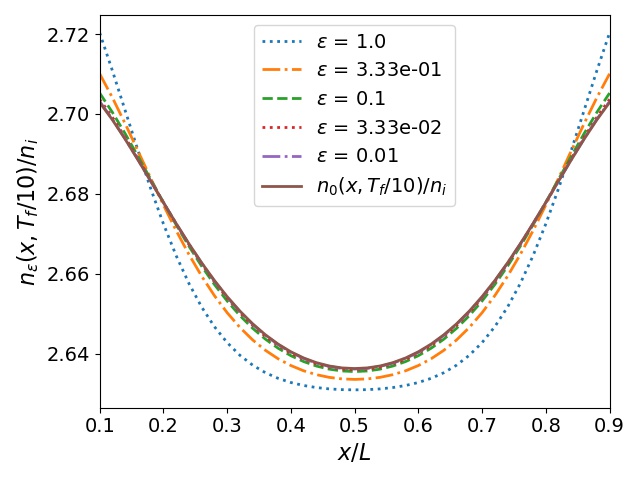}
\includegraphics[scale=0.45]{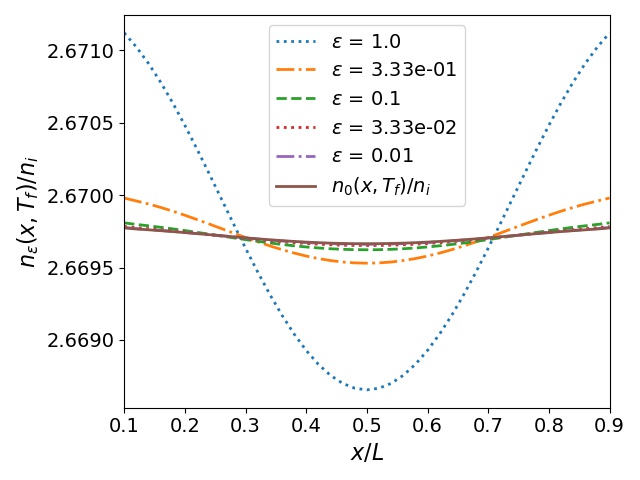}
\includegraphics[scale=0.45]{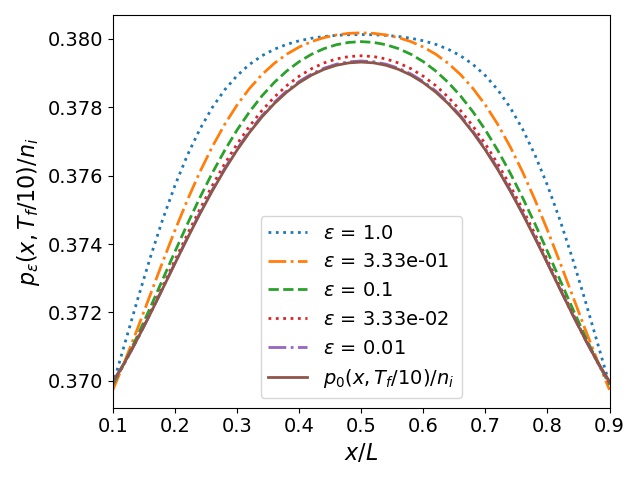}
\includegraphics[scale=0.45]{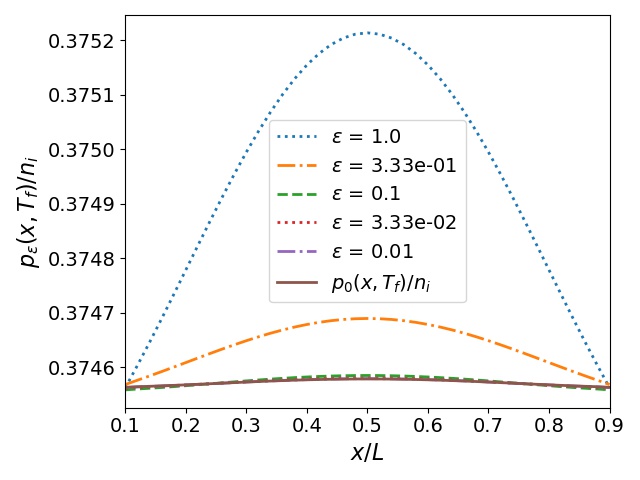}
\end{center}
\caption{Oxygen vacancy density (top row), electron density (middle row),
and hole density (bottom row) versus space at time $t=T_f/10$ (left column) 
and $t=T_f$ (right column) for various values of $\eps$ and the reduced problem.}
\label{fig.eps1}
\end{figure}

\begin{figure}[ht]
\begin{center}
\includegraphics[scale=0.45]{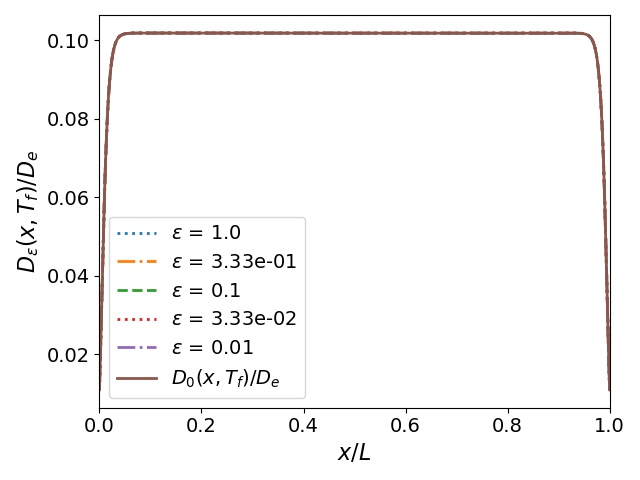}
\includegraphics[scale=0.45]{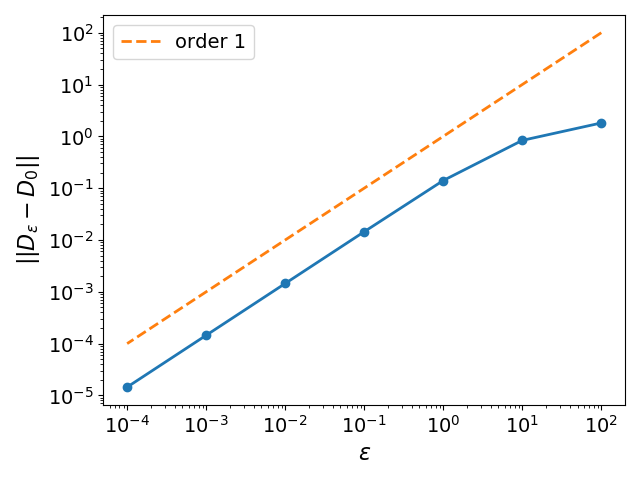}
\end{center}
\caption{Left: Oxygen vacancy density versus space at time $t=T_f$ 
for various values of $\eps$. Right: Difference of the oxygen vacancy densities 
$D_\eps$ and $D_0$ in the $L^1(\Omega\times(0,T_f))$ norm versus $\eps$.}
\label{fig.eps2}
\end{figure}

\subsection{Reduced system}

In the following, we focus on the reduced system \eqref{6.eq1}--\eqref{6.eq2}.
First, we choose a vanishing applied voltage ($U_0=U_L=0$)
and consider different values for the
dopant concentration at the electrode $D_e$. Figure \ref{fig.De} (left) shows the
spatial distribution of the quotient $D_0(x,T_f)/D_e$,
where $D_0(x,T_f)$ is close to the steady state. 
We observe a $U$-shape distribution with a
boundary layer near the electrodes. According to \cite{SBW09},
the layer comes from the fact that a large vacancy density gradient
near the electrode interfaces is required to compensate the strong
electrostatic attraction of ions to the image charge on both electrodes.
When $D_e/D^I < 1$, the vacancy density decreases away from the electrodes
and meet in the center of the device with a vanishing slope, and the shape is
inversed when $D_e/D^I>1$. 

\begin{figure}[ht]
\begin{center}
\includegraphics[scale=0.45]{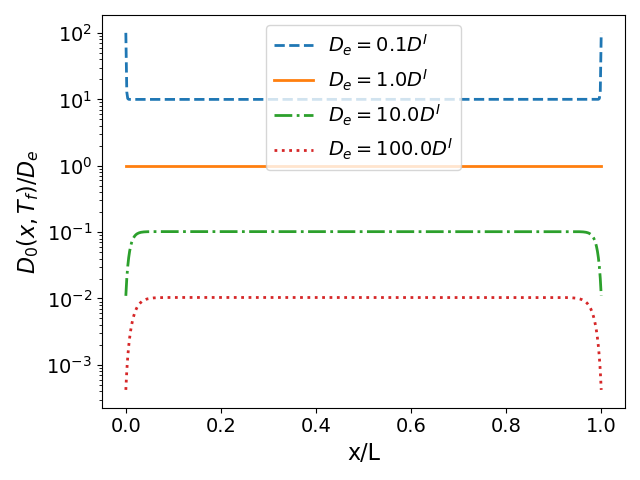}
\includegraphics[scale=0.45]{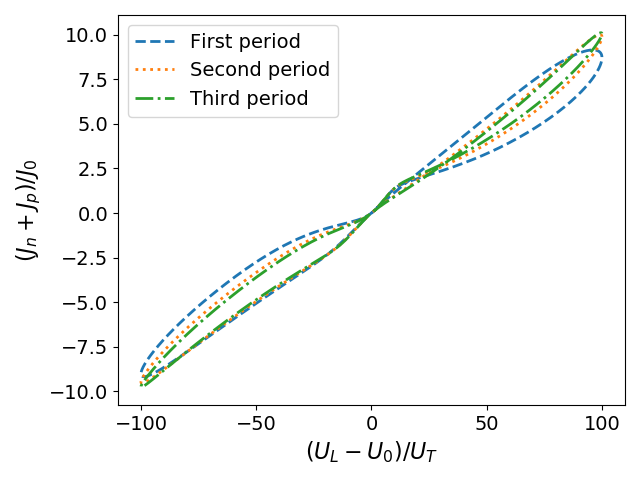}
\end{center}
\caption{Left: Rescaled vacancy density $D_0(x,T)/D_e$ for different values 
of $D_e/D^I$. Right: Current-voltage characteristics for three periods of a
sinuoidal applied voltage.}
\label{fig.De}
\end{figure}

For the following figure, we fix $D_e/D^I=10$. 
Figure \ref{fig.DV} shows the zero-bias potential 
$V(x,T_f)-V_{\rm bi}-V_{\rm applied}(x)$ 
and vacancy density at final time $t=T_f$ for various applied voltages 
$U_L-U_0$, scaled with the thermal voltage $U_T=26$\,mV. Here, we have set
$V_{\rm applied}(x)=(U_L-U_0)(x/L)-U_0$. The applied voltage produces a
potential barrier for the electrons; it causes the mobile vacancies to drift
and results in a complete vacancy depletion at the right side of 
the device. Similar results have been obtained in \cite[Figure 1]{SBW09}.

\begin{figure}[ht]
\begin{center}
\includegraphics[scale=0.45]{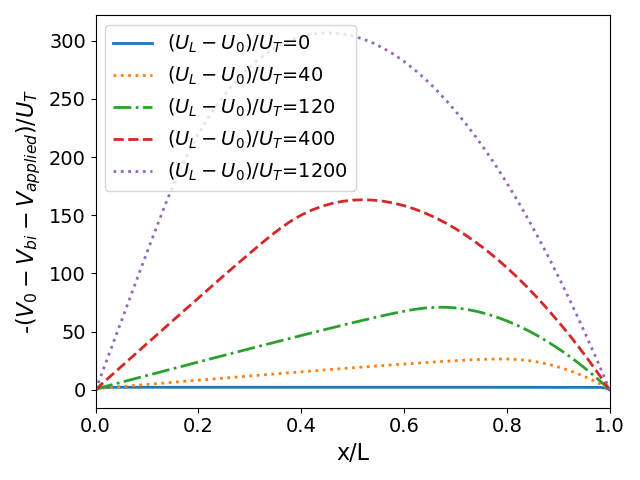}
\includegraphics[scale=0.45]{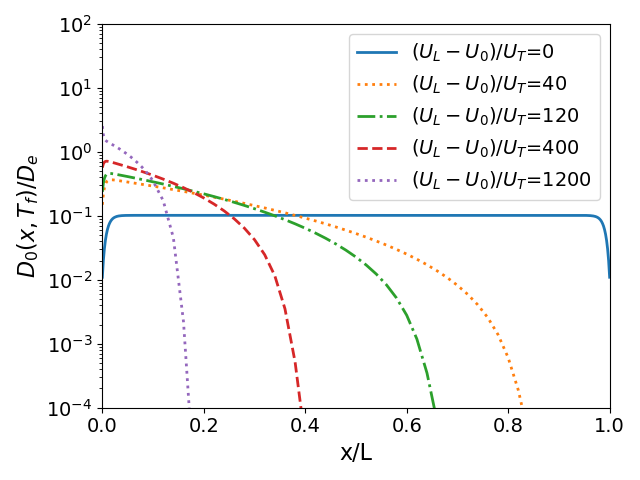}
\end{center}
\caption{Zero-bias potential $V(x,T_f)-V_{\rm bi}-V_{\rm applied}$ (left)
and oxygen vacancy density $D_0(x,T_f)$ (right) for various applied voltages.}
\label{fig.DV}
\end{figure}

Finally, we consider a sinusoidal applied voltage with $U_0=0$, 
$U_L(t)=100\sin(6\pi t/T_f)$, and $T_f=0.03$. 
The resulting dynamic current-voltage characteristics
($J_n(0)+J_p(0)$ versus $U_L-U_0$) are shown in Figure \ref{fig.De} (right). 
As in \cite{SBW09}, we observe a pinched hysteresis loop. 
This loop is a well-known
fingerprint of the ideal memristor introduced in \cite{Chu71}. 
The same applied potential leads to different current 
values at different times, which indicates that the device has a memory.
This confirms that the
drift-diffusion model is able to represent a memristive device.


\begin{appendix}
\section{Auxiliary lemmas}\label{sec.aux}

Introduce the functions $T_k(s)=\min\{k,s\}$ for $s\ge 0$, $k\geq1$ and
$$
  g_k(s) = \int_0^s\int_1^y\frac{dz}{T_k(z)}dy, \quad
	h_k(s) = \int_0^s\frac{dz}{\sqrt{T_k(z)}}, \quad s\ge 0.
$$

\begin{lemma}\label{lem.g}
It holds that $g_k(s) = s(\log s-1)$ for $0\le s<k$ and 
$g_k(s)\ge k(\log k-1) + (s-k)^2/(2k)$ for $s\ge k$.
\end{lemma}

\begin{proof}
We estimate for $0\le s\le k$,
$$
  g_k(s) = -\int_0^s\int_y^1\frac{dz}{T_k(z)}dy
	= -\int_0^s\int_y^1\frac{dz}{z}dy = s(\log s-1);
$$
and for $s\ge k$,
$$ 
  g_k(s) = \int_0^k\int_1^y\frac{dz}{T_k(z)}dy 	
  + \int_k^s\int_1^y\frac{dz}{T_k(z)}dy \ge k(\log k-1) + \frac{1}{2k}(s-k)^2,
$$
ending the proof.
\end{proof}

\begin{lemma}\label{lem.aux}
There exists $C>0$ such that for all $k>1$,
$$
  \sqrt{T_k(s)} \le C\big(1+\sqrt{|g_k(s)|}\big), \quad
	\sqrt{T_k(s)} \le Ch_k(s)\quad\mbox{for }s\ge 0.
$$
\end{lemma}

\begin{proof}
We first prove the inequality $h_k(s)^2 \le C(1+|g_k(s)|)$ 
and then $\sqrt{T_k(s)} \le h_k(s)/2$ for $s\ge 0$. 
Combining both inequalities shows the lemma. The second inequality follows from
\begin{align*}
  h_k(s) &= \int_0^s\frac{dy}{\sqrt{y}} = 2\sqrt{s} = 2\sqrt{T_k(s)}\quad
	\mbox{for }0<s<k, \\
	h_k(s) &= \int_0^k\frac{dy}{\sqrt{y}} + \int_k^s\frac{dy}{\sqrt{k}}
	\ge 2\sqrt{k} = 2\sqrt{T_k(s)}\quad\mbox{for }s\ge k.
\end{align*}
If $0<s<k$, we have shown in Lemma \ref{lem.g} that
$g_k(s)=s(\log s-1)$, and then $h_k(s)^2 \le C(1+|g_k(s)|)$ is equivalent to
$4s\le C(1+s|\log s-1|)$ for $0<s<k$, and this inequality is true for a suitable
$C>0$. If $s\ge k$, again by Lemma \ref{lem.g}, $h_k(s)^2 \le C(1+|g_k(s)|)$
follows from
$$
  \bigg(2\sqrt{k} + \frac{s-k}{\sqrt{k}}\bigg)^2 
	\le C\bigg(k(\log k-1) + \frac{(s-k)^2}{2k}\bigg),
$$
and this inequality is valid for a suitably chosen $C>0$ independent of $k$.
\end{proof}

Let $\xi>0$ and define $g_\xi(x) = \xi x(e^x-1)$ for $x\ge 0$ and its convex conjugate
$g_\xi^*(y) = \sup_{x>0}(xy-g_\xi(x))$ for $y\ge 0$. The following lemma
provides an upper bound for $g_\xi^*$.

\begin{lemma}\label{lem.conj}
The convex conjugate function of $g_\xi$ can be estimated as
$$
  g_\xi^*(y) \le \xi\frac{(\log(1+y/\xi))^2}{1+\log(1+y/\xi)}
	\bigg(1+\frac{y}{\xi}\bigg)\quad\mbox{for }y\ge 0.
$$
\end{lemma}

\begin{proof}
For given $y\ge 0$, let $\overline{x}(y)\ge 0$ be the unique solution to
$y=g'_\xi(\overline{x}(y))=\xi(1+\overline{x}(y))e^{\overline{x}(y)}-\xi$. Then
$$
  g_\xi^*(y) = \left\{\begin{array}{ll}
	\overline{x}(y)y-g_\xi(\overline{x}(y)) &\quad\mbox{for }y>g_\xi'(0)=0, \\
	0 &\quad\mbox{for }y=g_\xi'(0)=0.
	\end{array}\right.
$$
Furthermore, it follows from the definition of $\overline{x}(y)$ 
that $y/\xi \ge e^{\overline{x}(y)} - 1$ and hence,
$\overline{x}(y)\le\log(1+y/\xi)$. Therefore, since 
$(1+\overline{x}(y))e^{\overline{x}(y)} = 1+y/\xi$ by the definition of 
$\overline{x}(y)$, we have for $y\ge 0$,
\begin{align*}
  g_\xi^*(y) &= \overline{x}(y)y - \xi\overline{x}(y)(e^{\overline{x}(y)}-1)
	= \xi\overline{x}(y)\bigg(1+\frac{y}{\xi}-e^{\overline{x}(y)}\bigg) \\
	&= \xi\overline{x}(y)\bigg(1+\frac{y}{\xi}-\frac{1+y/\xi}{1+\overline{x}(y)}\bigg)
	= \frac{\overline{x}(y)^2}{1+\overline{x}(y)}(y+\xi)
	\le \frac{(\log(1+y/\xi))^2}{1+\log(1+y/\xi)}(y+\xi),
\end{align*}
which shows the lemma.
\end{proof}

We continue with some Gagliardo--Nirenberg (type) inequalities.

\begin{lemma}[Gagliardo--Nirenberg]\label{lem.GN}
Let $\Omega\subset\R^d$ ($d\ge 1$) be a bounded domain with Lip\-schitz boundary
and let $q\le 2d/(d-2)$ if $d>2$ and $q<\infty$ if $d=2$.
Then for all $\delta>0$, there exist $C>0$ and $C(\delta)>0$
such that for all $u\in H^1(\Omega)$,
\begin{align}
  \|u\|_{L^q(\Omega)} &\le C\|u\|_{H^1(\Omega)}^\theta\|u\|_{L^1(\Omega)}^{1-\theta}, 
	\label{GN} \\
	\|u\|_{L^q(\Omega)} &\le \delta\|u\|_{H^1(\Omega)}^\theta
	\|u\log|u|\|_{L^1(\Omega)}^{1-\theta} + C(\delta)\|u\|_{L^1(\Omega)}, \label{GN2}
\end{align}
where $\theta = 2d(q-1)/((d+2)q)$. In two space dimensions, we have $\theta=1-1/q$.
\end{lemma}

\begin{proof}
Inequality \eqref{GN} is the standard Gagliardo--Nirenberg inequality.
Using this inequality, inequality \eqref{GN2} can be proved as in
\cite[(22)]{BHN94}, where the inequality was shown for $q=3$; also see
\cite[(1.9)]{GlHu97}.
\end{proof}

We recall the following regularity result, valid in two space dimensions and
proved in \cite{Gro94}; also see \cite[Lemma 3.1]{GlHu97}. 

\begin{lemma}[Regularity for the Poisson equation]\label{lem.vinfty}
Let $\Omega\subset\R^2$ satisfy Assumption (A1), and
let $v\in H^1(\Omega)$ be the unique solution to $\Delta v=f$ in $\Omega$,
$v=\overline{v}$ on $\Gamma_D$, and $\na v\cdot\nu=0$ on $\Gamma_N$.
There exist $r_0>2$ and $C>0$ such that
\begin{align}\label{v.infty}
  \|v\|_{L^\infty(\Omega)} &\le C\big(\|f\log|f|\|_{L^1(\Omega)}
	+ g(\|v\|_{H^1(\Omega)}) + 1\big), \\
	\|v\|_{W^{1,r_0}(\Omega)} &\le C\big(\|f\|_{L^{2r_0/(r_0+2)}(\Omega)}
	+ g(\|v\|_{H^1(\Omega)}) + 1\big), \nonumber 
\end{align}
where $g$ is a continuous increasing function.
\end{lemma}

The following lemma follows from the Alikakos iteration method. A proof
can be found in \cite{HoJu20} for homogeneous boundary conditions. The
proof is the same for no-flux and mixed boundary conditions.

\begin{lemma}\label{lem.infty}
Let $\Omega\subset\R^d$ ($d\ge 1$) satisfy Assumption (A1) and let
$u^{q/2}\in L^2(0,T;H^1(\Omega))\cap L^\infty(0,T;L^2(\Omega))$
for all $q\in\N$ with $q\ge 2$ with $u\ge 0$ in $\Omega\times(0,T)$,
$u(0)=0$ in $\Omega$, and either
$u=0$ on $\Gamma_D$, $\na u\cdot\nu=0$ on $\Gamma_N$, or $u=0$ on $\pa\Omega$,
or $\na u\cdot\nu=0$ on $\pa\Omega$. Assume that there are
constants $K_0$, $K_1$, $K_2>0$ and $\alpha$, $\beta\ge 0$
such that for all $q\ge 2$, $t\in(0,T)$,
\begin{equation*}
  \int_\Omega e^{t}u(t)^q dx
	+ K_0\int_0^t\int_\Omega e^{s}|\na u^{q/2}|^2 dxds
	\le K_1 q^\alpha\int_0^t\int_\Omega e^{s}u^q dxds
	+ K_2 q^\beta e^{t}.
\end{equation*}
Then
$$
  u(t) \le K_3\big(\|u\|_{L^\infty(0,T;L^1(\Omega))} + 1\big)
	\quad\mbox{in }\Omega,\ t\in(0,T),
$$
where $K_3$ depends only on $\alpha$, $\beta$, $d$, $\Omega$, $K_0$, $K_1$, and $K_2$.
\end{lemma}

\end{appendix}


\end{document}